\documentclass[reqno,12pt, oneside]{amsart}

\usepackage{enumerate}
\usepackage[nobysame]{amsrefs}
\usepackage{enumitem}
\usepackage{amssymb}
\usepackage{amsmath}
\usepackage{amsthm}
\usepackage{mathtools}
\usepackage{amscd}
\usepackage{microtype}
\usepackage[right=2cm,left=2cm, top=2.5cm, bottom=2.5cm,includehead]{geometry}
\usepackage{amsfonts}

\usepackage{tgpagella}
\usepackage{eulervm}

\usepackage[backgroundcolor=lightgray]{todonotes}

\theoremstyle{plain}
\newtheorem{theorem}{Theorem}
\newtheorem{corollary}[theorem]{Corollary}
\newtheorem{lemma}[theorem]{Lemma}
\newtheorem{proposition}[theorem]{Proposition}

\theoremstyle{remark}
\newtheorem{remark}[theorem]{Remark}
\newtheorem{example}[theorem]{Example}

\theoremstyle{definition}
\newtheorem{definition}[theorem]{Definition}
\newtheorem*{notation}{Notation}

\DeclarePairedDelimiter{\abs}{\lvert}{\rvert}
\DeclarePairedDelimiter{\vnorm}{\langle\!\langle}{\rangle\!\rangle}
\DeclarePairedDelimiter{\norm}{\lVert}{\rVert}
\DeclarePairedDelimiter{\set}{\lbrace}{\rbrace}
\DeclarePairedDelimiterX{\dset}[2]{\lbrace}{\rbrace}{#1\;\delimsize|\;#2}

\DeclareMathOperator{\diam}{diam}

\newcommand{\ball}{{\overline{B}}}

\newcommand{\N}{\mathbb N}

\DeclareMathOperator{\BLip}{BLip}
\DeclareMathOperator{\Lip}{Lip}
\DeclareMathOperator{\lip}{lip}

\allowdisplaybreaks

\title[Compactness in Lipschitz spaces and around]{Compactness in Lipschitz spaces and around}

\author{Jacek Gulgowski}
\address[J.~Gulgowski]{Institute of Mathematics\\
Faculty of Mathematics, Physics and Informatics\\
University of Gda\'nsk\\
80-308 Gda\'nsk\\
Poland}
\email[J.~Gulgowski]{dzak@mat.ug.edu.pl}

\author{Piotr Kasprzak}
\address[P. Kasprzak]{Department of Nonlinear Analysis and Applied Topology\\
Faculty of Mathematics and Computer Science\\
  Adam Mickiewicz University in Pozna\'n\\
  ul.\ Uniwersytetu Pozna\'nskiego 4\\
  61-614 Pozna\'n\\
  Poland}
\email[P.~Kasprzak]{kasp@amu.edu.pl}

\author{Piotr Ma\'ckowiak}
\address[P. Ma\'ckowiak]{Department of Nonlinear Analysis and Applied Topology\\
Faculty of Mathematics and Computer Science\\
  Adam Mickiewicz University in Pozna\'n\\
  ul.\ Uniwersytetu Pozna\'nskiego 4\\
  61-614 Pozna\'n\\
  Poland}
\email[P.~Ma\'ckowiak]{piotr.mackowiak@amu.edu.pl}


\date{\today}

\begin{document}
\begin{abstract}
The aim of the paper is to characterize (pre)compactness in the spaces of Lipschitz/H\"older continuous mappings, which act from an arbitrary (not necessarily compact) metric space to a normed space. To this end some extensions and generalizations of already existing compactness criteria for the spaces of bounded and continuous mappings with values in normed spaces needed to be established. Those auxiliary results, which are interesting on their own since they use a concept of equicontinuity not seen in the literature, are based on an abstract compactness criterion related to the recently introduced notion  of an equinormed set.
\end{abstract}

\maketitle

\section{Introduction}

The importance of the spaces of Lipschitz and H\"older continuous functions in mathematics and other sciences cannot be overestimated. This introduction is too short even to list all the theories where such functions play a key role, let aside describe them in detail. Therefore, we will limit ourselves to mentioning only those instances which are related to compactness, i.e., the main object of our study. One of the oldest examples of an application of H\"older continuous functions dates back to 1930s and to Leray and Schauder, who used them in the theory of partial differential equations. (For more details, we refer the reader to the well-known monographs \cites{evans, gilbarg_trudinger}.) When talking about applications of the spaces of H\"older continuous functions one should not forget singular operators appearing in the fractional calculus. It turns out that such spaces are natural domains for the Riemann\--Liouville integral operator of fractional order or one of its most notable representatives, that is, the Abel operator (for more details, see~\cite{GV} and~\cite{samko}*{Chapters 3.1 and 3.2}). The other noteworthy situation where Lipschitz spaces emerge naturally is the Kantorovich–Rubinstein metric defined in the space of Borel measures on a compact metric space. These metrics are used to define the weak sequential convergence of measures and are applied in the treatment of the Monge--Kantorovich mass transport problem. (The introductory discussion on the subject, from the perspective of the spaces of Lipschitz continuous functions, may be found in~\cite{cobzas_book}*{Section~8.4}.)

Because of their applications -- in the previous paragraph we barely touched this topic -- Lipschitz and H\"older continuous functions constantly enjoy great interest among researchers. Each year numerous scientific articles and chapters are devoted just to studying their properties. Even whole books on Lipschitz and H\"older functions are written. Among the recent ones let us mention two monographs: one by Weaver, whose second (extended) edition was published in 2018 (see~\cite{weaver}), and the other by Cobza{\c{s}} \emph{et al.} published in 2019 (see~\cite{cobzas_book}). In those monographs one may find plenty of profound ideas, perspectives and applications related to the spaces of Lipschitz and H\"older continuous functions. What they lack, however, is a (strong) compactness criterion in such spaces. No matter how hard we tried, we could not find one in the literature, although, as it turned out, in the case of real-valued functions it had stayed hidden in a plain sight for more than half a century. In a paper of Cobza{\c{s}} from 2001 it is even stated: ``\emph{\textup[\ldots\textup] apparently there is no compactness criterion in spaces of H\"older functions, and some criteria given in the literature turned to be false}'' (see~\cite{cobzas2001}*{p.~9}).  

Over the last years several attempts to find a characterization of compactness in the spaces of Lipschitz and/or H\"older continuous mappings have been made. For example, in~\cite{cobzas2001} Cobza{\c{s}} proved a precompactness criterion in the space  consisting of those vector-valued Lipschitz maps that, roughly speaking, are restrictions of continuously differentiable function. In~\cite{banas_nalepa} the authors gave a sufficient condition for relative compactness of subsets of the space of H\"older continuous real-valued functions; this result was later repeated in~\cite{banas_nalepa2}.  It should be underlined here that this condition, which we will refer to as \emph{uniform local flatness}, is far from being necessary. Many natural Lipschitz and H\"older continuous functions (like, for example, $f \colon [0,1] \to \mathbb R$ given by $f(x)=x$) simply do not satisfy it. It turns out, however, that uniform local flatness can be used to characterize precompactness in the so-called \emph{little Lipschitz \textup(H\"older\textup) spaces}. (For the appropriate definitions and the compactness criterion in little Lipschitz spaces see Section~\ref{sec:55} below.) This was already known to Johnson in the real-valued setting (see~\cite{johnson}*{Theorem~3.2}) and to Garc\'{\i}a-Lirola \emph{et al.} in a slightly different context of little Lipschitz functions that are continuous with respect to a topology which needs not to come from the metric of the underlying space (see~\cite{GPR}*{Lemma~2.7}). Finally, based on the ideas of Bana\'s and Nalepa, in~\cite{Saiedinezhad} Saiedinezhad used the uniform local flatness to provide a sufficient condition for relative compactness of a non-empty subset of the space $C^{k,\alpha}(X,\mathbb R)$. (Let us recall that $C^{k,\alpha}(X,\mathbb R)$ consists of real-valued multivariate functions that are $k$-times continuously differentiable and whose partial derivatives of order $k$ are H\"older continuous with exponent $\alpha$; here $X$ is a compact subset of $\mathbb R^n$.) Clearly, this result cannot be necessary. However, Saiedinezhad used it to introduce and study a measure of non-compactness in the space $C^{k,\alpha}(X,\mathbb R)$.

When thinking about possible approaches to proving a compactness criterion in a given normed space, two obvious methods come to mind: direct and indirect one. The former one is self-explanatory. The latter one often consists of two steps: finding another -- simpler -- space (linearly) isomorphic with our given space and proving a compactness criterion in this simpler space. It turns out that for $\Lip_\varphi(X,E)$ as the simpler space we can take the product of $E$ and the space of bounded (or bounded and continuous) functions. (Here, and in the sequel, $\Lip_\varphi(X,E)$ stands for the space of H\"older continuous mappings defined on a metric space $(X,d)$ and with values in a normed space $E$.) The isomorphism can be constructed using the so-called de Leeuw’s map. There is only one ``little'' detail that needs to be taken care of. To a given H\"older continuous function $f \colon X \to E$ the de Leeuw’s map assigns the bounded continuous mapping $\Phi(f)$, which for distinct $x,y \in X$ is defined by the formula $\Phi(f)(x,y)=(f(x)-f(y))/\varphi(d(x,y))$; here $\varphi$ is the so-called comparison function. (More details on comparison functions and de Leeuw’s map can be found in Section~\ref{sec:compact_in_lip} below and in the references given therein.) So, as we can see, the function $\Phi(f)$ is not defined on the whole product $X \times X$, but on its proper subset $(X\times X) \setminus \dset{(x,x)}{x \in X}$, which may not be compact even if $X$ is. And the compactness criteria for $B(Y,E)$ or $C(Y,E)$, in general, are not easy to come by, especially if $E$ is infinite-dimensional and $Y$ is an arbitrary non-empty set or a topological space. Here, and in the rest of this paper, by $B(Y,E)$ we denote the space of bounded functions defined on the non-empty set $Y$ with values in the normed space $E$, and by $C(Y,E)$ -- its subspace of continuous functions. Of course, in the case of $C(Y,E)$ we assume that $Y$ is at least a topological space. Note also that we do not require $Y$ to be bounded or compact.

One example of such a criterion for real-valued functions can be found in the celebrated monograph of Dunford and Schwarz \cite{DS}. Using an approach of Phillips (see~\cite{phillips} as well as~\cite{eveson}) it is shown that a non-empty and bounded subset $A$ of $B(Y,\mathbb R)$ or $C(Y,\mathbb R)$ is relatively compact if and only if for every $\varepsilon>0$ there is a finite cover $U_1,\ldots,U_N$ of $Y$ and points $t_i \in U_i$ such that $\sup_{f \in A}\sup_{x \in U_i}\abs{f(x)-f(t_i)}\leq \varepsilon$ for $i=1,\ldots,N$ (cf.~\cite{DS}*{Theorem~IV.5.6, p.~260 and Theorem~IV.6.5, p.~266}); here $Y$ is either an arbitrary (non-empty) set in the case of $B(Y,\mathbb R)$, or a topological space in the case of $C(Y,\mathbb R)$. (The condition appearing above is equivalent to the condition (DS) discussed in the present article. Because of that, we will refer to it as the (DS) condition as well.)

In our study we have been interested in a compactness criterion for the space of H\"older continuous mappings taking values in an arbitrary normed space $E$. (The space $E$ may be over the field of either real or complex numbers.) What may come as a surprise (having in mind, for example, the Arzel\`a--Ascoli theorem) is that the mere replacement of absolute values with norms in the results of Dunford and Schwarz simply does not work. It is not difficult to provide examples of compact (or even finite) subsets of the spaces $B(Y, E)$ or $C(Y,E)$, with $Y$ being a metric space, which do not satisfy the condition (DS) – for more details and examples see Section~\ref{sec:discussion} below. Because of that, we had to find a completely new (not based on Phillip's results) approach to establishing compactness criteria in $B(Y,E)$ and/or $C(Y,E)$. We also needed a completely new condition equivalent to precompactness in those spaces, which is more general than the condition (DS). (Note that we do not assume that $E$ is complete, hence we speak of precompactness rather than relative compactness.) Our method of proof is based on a version of a very recent abstract compactness criterion proved in~\cite{GKM}. This result, roughly speaking, allows to translate precompactness to a certain inequality involving the norm of the considered space. We should also mention that the ideas of Ambrosetti  (see~\cite{Ambrosetti}*{Section~2} and~\cite{BBK}*{Lemma~1.2.8}) were of much help in our study. It is worth underlining here that our approach is more elementary than the one presented in the book of Dunford and Schwarz~\cite{DS} (in a sense that it does not use advanced concepts from functional analysis or topology). However, it allows not only to prove precompactness criteria in $B(Y, E)$, $C(Y,E)$ and various spaces of Lipschitz maps in the general case when $E$ is an arbitrary normed space, but also to draw some additional conclusion about the regularity of sets the $U_i$ covering the space $Y$. (It turns out that those sets might be chosen either closed or open, if the function space in question consists of continuous maps.) We refer the reader to Section~\ref{sec:discussion} below, where we provide a detailed discussion of the interplay between the dimension of the target space $E$ and various conditions related to precompactness, which are studied in the paper.

As already mentioned, using the results for $B(Y, E)$ and $C(Y,E)$, in the second part of the paper we prove precompactness criteria for various spaces of Lipschitz continuous functions defined on an arbitrary metric space. In general, we do not have to assume that the domain $X$ is either bounded or compact. It is also important to notice that in the case of unbounded domains  the space $\Lip_\varphi(X,E)$ and its linear subspace $\BLip_\varphi(X,E)$, consisting of bounded maps, do not coincide. Further, the default norm on $\BLip_\varphi(X,E)$ is stronger than the norm inherited from $\Lip_\varphi(X,E)$. Nevertheless, we are able to formulate and discuss (pre)compactness conditions for both spaces $\Lip_\varphi(X,E)$ and $\BLip_\varphi(X,E)$. We even study precompactness in another important space of Lipschitz continuous mappings, that is, in $\Lip_0(X,E)$. And, we do not stop there. We analyse the main condition appearing in those results, and show that what really matters for compact domains is a certain uniform behaviour of functions for arguments which are close to each other. This allows us to formulate an equivalent -- localized -- version of the compactness criterion in $\Lip_\varphi(X,E)$, which we later use to obtain a compactness criterion in the little Lipschitz space $\lip_\varphi(X,E)$.

\section{Preliminaries}
The main goal of this section is to introduce the notation and conventions used throughout the paper and to recall some basic facts. We will use default symbols to denote the classical spaces and norms. Let $(E, \norm{\cdot}_E)$ be a normed space over the field of either real or complex numbers. By $B(X,E)$ we will denote the normed space of all bounded $E$-valued maps defined on a non-empty set $X$, endowed with the supremum norm $\norm{f}_\infty:=\sup_{x \in X}\norm{f(x)}_E$. Assuming that $X$ is a topological space, by $C(X,E)$ we will denote the normed subspace of $B(X,E)$ consisting of all continuous maps $f \colon X \to E$. Given a set $A\subseteq B(X,E)$ and a point $x \in X$, let $A(x):=\dset{f(x)}{f \in A}$; we will call such a set a \emph{section} of $A$ at $x$. We also set $A(X):=\bigcup_{x \in X}A(x)$. The open and closed balls in a metric space $X$ with center at $x\in X$ and radius $r>0$ will be denoted by $B_X(x,r)$ and $\ball_X(x,r)$, respectively. Finally, if $A, B$ are two non-empty subsets of a linear space, then $A+B:=\dset{a+b}{a\in A,\ b\in B}$ and $A-B:=\dset{a-b}{a\in A,\ b\in B}$. For functions $f,g \colon X \to E$, mapping a set $X$ into a normed space $E$, the difference $f-g$ will be always defined pointwise.

\subsection{Precompactness and relative compactness} 

The definition of precompactness is well-known. However, it is sometimes mistaken with relative compactness. So, to avoid any ambiguity, let us state the definition explicitly. A metric space $X$ is \emph{precompact}, if its completion is compact. Equivalently, $X$ is precompact, if every sequence in $X$ contains a Cauchy subsequence. In the sequel, we will also use the fact that the classes of precompact and totally bounded metric spaces coincide. Let us recall that a metric space $X$ is \emph{totally bounded}, if for each $\varepsilon>0$ there is a finite collection of points $x_1,\ldots,x_n \in X$ such that $X= \bigcup_{i=1}^n\ball_X(x_i,\varepsilon)$; such a collection is often called an \emph{$\varepsilon$-net}. A subset $A$ of a metric space $X$ is \emph{precompact}, if it is precompact as a metric space itself (with the metric inherited from $X$).

A non-empty subset $A$ of a metric space $X$ is called \emph{relatively compact in $X$}, if the closure of $A$ with respect to $X$ is compact. Equivalently, $A$ is relatively compact in $X$, if each sequence in $A$ contains a subsequence convergent to an element in $X$.

In complete metric spaces the above notions connected with compactness coincide. For more information concerning precompact and relatively compact metric spaces see~\cite{BBK}*{Section~1.1} or~\cite{MV}*{Chapter~4}.

\subsection{Abstract compactness criterion}
The proofs of the main results of the paper will be based on a modification of a new abstract compactness criterion that was presented in~\cite{GKM}. 

Let $(E,\norm{\cdot}_E)$ be a normed space endowed with a family $\bigl\{\norm{\cdot}_i\bigr\}_{i \in I}$ of semi-norms on $E$. We assume that this family  satisfies the following two conditions:
\begin{enumerate}[label=\textbf{\textup{(A\arabic*)}}]
 \item\label{i} $\norm{x}_E=\sup_{i \in I}\norm{x}_i$ for $x \in E$,
 \item\label{ii} for every $i,j \in I$ there is an index $k \in I$ such that $\norm{x}_i \leq \norm{x}_k$ and $\norm{x}_j \leq \norm{x}_k$ for $x \in E$; in other words, the family $\bigl\{\norm{\cdot}_i\bigr\}_{i \in I}$ forms a directed set. 
\end{enumerate}
A non-empty subset $A$ of $E$ is called \emph{equinormed} (\emph{with respect to the family $\bigl\{\norm{\cdot}_i\bigr\}_{i \in I}$}), if for every $\varepsilon>0$ there exists $k \in I$ such that $\norm{x}_E\leq \varepsilon + \norm{x}_k$ for all $x \in A$.

Now, we are in position to state the afore-mentioned abstract compactness criterion. Although this is a slight modification of~\cite{GKM}*{Theorem~19}, we will skip the proof completely, as it is identical to the original one. 

\begin{theorem}\label{thm:compactness_ver2}
Let $(E,\norm{\cdot}_E)$ be a normed space equipped with a family $\bigl\{\norm{\cdot}_i\bigr\}_{i \in I}$ of semi-norms satisfying conditions~\ref{i} and~\ref{ii}. A non-empty subset $A$ of $E$ is precompact if and only if
\begin{enumerate}[label=\textup{(\roman*)}]
 \item\label{it:compactness_ver2_i} the set $A-A$ is equinormed, and
 \item\label{it:compactness_ver2_ii} for each sequence $(x_n)_{n \in \mathbb N}$ of elements of $A$ there is a subsequence $(x_{n_k})_{k \in \mathbb N}$ which is Cauchy with respect to each semi-norm $\norm{\cdot}_i$.
\end{enumerate}
\end{theorem}

\begin{remark}\label{rem:boundedness}
In Theorem~\ref{thm:compactness_ver2} we need not require the set $A$ to be bounded, as its boundedness follows easily from other assumptions. It is clear that if $A$ is precompact, it is also bounded. On the other hand, suppose that the set $A$ satisfies the conditions~\ref{it:compactness_ver2_i} and~\ref{it:compactness_ver2_ii} above, but is not bounded. Then, there is a sequence $(y_n)_{n \in \mathbb N}$ in $A$ such that $\norm{y_n}_E \to +\infty$ as $n \to +\infty$. Let $(y_{n_k})_{k \in \mathbb N}$ be its subsequence which is Cauchy with respect to each semi-norm $\norm{\cdot}_i$. In particular, $(y_{n_k})_{k \in \mathbb N}$ is bounded in each semi-norm, that is, $M_i:=\sup_{k \in \mathbb N}\norm{y_{n_k}}_i<+\infty$ for every $i \in I$. As the set $A-A$ is equinormed, there is an index $j \in I$ such that $\norm{\xi-\eta}_E \leq 1+\norm{\xi-\eta}_j$ for all $\xi,\eta \in A$. And then we have $\norm{y_{n_k}-y_{n_1}}_E \leq 1+ \norm{y_{n_k}-y_{n_1}}_j \leq 1+2M_j$ for $k \in \mathbb N$. This means that $\sup_{k \in \mathbb N}\norm{y_{n_k}}_E<+\infty$, which is impossible.
\end{remark}

\section{Compactness in $B(X,E)$}

In Section~\ref{sec:compactness_B_X_E} we provide a full characterization of precompact subsets of the space of bounded mappings with values in a normed space. The proof of this result will not be a mere rewriting of the argument used in the case of real-valued functions. (Here, talking about the real-valued setting, we mainly think about~\cite{DS}*{Theorem~6, p.~260}.) A completely new approach is needed, partly because in the vector-valued case 
the condition imposed on a non-empty set $A\subseteq B(X,E)$ that guarantees its precompactness is different from the ``classical'' condition considered in~\cite{DS} (cf. condition (DS) below). Our proof will be based on Theorem~\ref{thm:compactness_ver2}. We will therefore begin Section~\ref{sec:compactness_B_X_E} by defining an appropriate family of semi-norms on $B(X,E)$. Next, we will introduce a new condition (called the condition (B)) that is the main ingredient of our precompactness criterion. Finally, we will state and prove the criterion itself.

In Section~\ref{sec:discussion_of_B} we discuss the condition (B) little bit more. We will provide examples illustrating that it cannot be easily simplified. We will also show that in the case of real-valued functions it agrees with the condition presented in~\cite{DS}*{Theorem~6, p.~260}.

\subsection{Compactness criterion in $B(X,E)$}\label{sec:compactness_B_X_E} 
Let us fix a non-empty set $X$ and a normed space $(E,\norm{\cdot}_E)$. Furthermore, let $\mathcal F$ be the family of all non-empty and finite subsets of $X$. For any $Y \in \mathcal F$ set
\begin{equation}\label{eq:norms}
\norm{f}_Y:=\sup_{x \in Y}\norm{f(x)}_E. 
\end{equation}
It is clear that the family of semi-norms $\set[\big]{\norm{\cdot}_Y}_{Y \in \mathcal F}$ satisfies the conditions~\ref{i} and \ref{ii}.

Let us introduce a condition which plays a key role in our considerations. 

\begin{definition}
A non-empty set $A$ in $B(X,E)$ is said to satisfy the condition (B), if for every $\varepsilon>0$ there is a finite cover $U_1,\ldots,U_N$ of $X$ such that for every $i \in \{1,\ldots,N\}$ and every pair $x, y \in U_i$ for all $f \in A$ we have $\abs[\big]{\norm{f(x)}_E - \norm{f(y)}_E} \leq \varepsilon$.
\end{definition}

Now, we are ready to state the precompactness criterion in $B(X,E)$.

\begin{theorem}\label{thm:compactness_BXE}
A non-empty subset $A$ of $B(X,E)$ is precompact if and only if the set $A-A$ satisfies the condition \textup{(B)} and for every $x \in X$ the sections $A(x)$ are precompact.
\end{theorem}

\begin{remark}\label{rem:relatively_compact}
If $E$ is a Banach space, then the space $B(X,E)$ is clearly complete. Therefore, the classes of its relatively compact and precompact subsets coincide. Hence, in such a case Theorem~\ref{thm:compactness_BXE} provides necessary and sufficient conditions for relative compactness of a non-empty subset of $B(X,E)$.
\end{remark}

Theorem~\ref{thm:compactness_BXE} is a consequence of the following three lemmas and Theorem~\ref{thm:compactness_ver2}.

\begin{lemma}\label{lem:5}
Let $A$ be a non-empty subset of $B(X,E)$ such that the algebraic difference $A-A$ satisfies the condition \textup{(B)}. Moreover, assume that the sections $A(x)$ are precompact for every $x \in X$. Then, each sequence $(f_n)_{n \in \mathbb N}$ of elements of the set $A$ contains a subsequence $(f_{n_k})_{k\in \mathbb N}$ which is Cauchy with respect to each semi-norm $\norm{\cdot}_Y$ given by~\eqref{eq:norms}. 
\end{lemma}

\begin{proof}
For each $i \in \mathbb N$ choose finitely many sets $U_1^i,\ldots,U_{N_i}^i$ satisfying the condition (B) with $\varepsilon=\frac{1}{i}$. In each set $U_j^i$ let us also fix a point $t_j^i$ and define $T:=\dset{t_j^i \in X}{\text{$i \in \mathbb N$ and $j=1,\ldots,N_i$}}$.

Now, let $(f_n)_{n \in \mathbb N}$ be a given sequence of elements of $A$. By the classical diagonal argument and the fact that all the sections $A(x)$ are precompact, we may select a subsequence  $(f_{n_k})_{k \in \mathbb N}$ which is Cauchy at every point $t \in T$. We will show that this subsequence is Cauchy at every $x \in X$. Take an arbitrary $x \in X$, fix $\varepsilon>0$ and choose $i \in \mathbb N$ such that $\frac{1}{i}\leq \frac{1}{2}\varepsilon$. As the sets $U_1^i,\ldots,U_{N_i}^i$ cover $X$, the point $x$ must belong to at least one of them; denote this set by $U_j^i$. The sequence $(f_{n_k}(t_j^i))_{k \in \mathbb N}$ is Cauchy in $E$, so there is an index $K$ such that $\norm{f_{n_k}(t_j^i) - f_{n_l}(t_j^i)}_E \leq \frac{1}{2}\varepsilon$ for all $k,l \geq K$. Then, in view of the condition (B), for all such indices $k,l$ we have $\norm{f_{n_k}(x)-f_{n_l}(x)}_E=\norm{(f_{n_k}-f_{n_l})(x)}_E - \norm{(f_{n_k}-f_{n_l})(t_j^i)}_E + \norm{f_{n_k}(t_j^i) - f_{n_l}(t_j^i)}_E \leq \varepsilon$. This proves that the sequence  $(f_{n_k}(x))_{k \in \mathbb N}$ is Cauchy.
\end{proof}

\begin{lemma}
Let $A$ be a non-empty subset of $B(X,E)$ such that $A-A$ satisfies the condition \textup{(B)}. Then, the algebraic difference $A-A$ is equinormed \textup(with respect to the family of semi-norms $\{\norm{\cdot}_Y\}_{Y \in \mathcal F}$ defined by~\eqref{eq:norms}\textup).
\end{lemma}

\begin{proof}
Similar to what we did in the proof of Lemma~\ref{lem:5} for each $i \in \mathbb N$ we choose finitely many sets $U_1^i,\ldots,U_{N_i}^i$ satisfying the condition (B) with $\varepsilon=\frac{1}{i}$. In each set $U_j^i$ we fix a point $t_j^i$. And, finally, we define $T:=\dset{t_j^i \in X}{\text{$i \in \mathbb N$ and $j=1,\ldots,N_i$}}$.

Now, let us fix $\varepsilon>0$ and let $i \in \mathbb N$ be such that $\frac{1}{i}\leq \frac{1}{2}\varepsilon$. Moreover, let $Y:=\{t_1^i,\ldots,t_{N_i}^i\}$. If $f,g \in A$, then there exists a point $x_\ast \in X$ such that $\norm{f-g}_{\infty} \leq \frac{1}{2}\varepsilon + \norm{(f-g)(x_\ast)}_E$. But $x_\ast$ belongs to one of the sets $U_1^i,\ldots,U_{N_i}^i$, say $U_j^i$, and so
\begin{align*}
 \norm{f-g}_{\infty} &\leq \tfrac{1}{2}\varepsilon + \norm{(f-g)(x_\ast)}_E\\
  & = \tfrac{1}{2}\varepsilon + \norm{(f-g)(x_\ast)}_E - \norm{(f-g)(t_j^i)}_E + \norm{(f-g)(t_j^i)}_E\\
	& \leq \varepsilon + \norm{f-g}_Y.
	\end{align*}
Therefore, the set $A-A$ is equinormed.	
\end{proof}

\begin{lemma}\label{lem:7}
Let $A$ be a non-empty subset of $B(X,E)$ such that $A-A$ is equinormed \textup(with respect to the family of semi-norms $\{\norm{\cdot}_Y\}_{Y \in \mathcal F}$ defined by~\eqref{eq:norms}\textup). Moreover, assume that each sequence $(f_n)_{n \in \mathbb N}$ of elements of the set $A$ contains a subsequence $(f_{n_k})_{k\in \mathbb N}$ which is Cauchy with respect to each semi-norm $\norm{\cdot}_Y$. Then, $A-A$ satisfies the condition \textup{(B)} and the sections $A(x)$ for $x \in X$ are precompact. 
\end{lemma}

\begin{proof}
Note that for any $x\in E$ the sequence $(f_n)_{n \in \mathbb N}$ in $B(X,E)$ is Cauchy with respect to the semi-norm induced by the singleton $Y:=\{x\}$ if and only if the sequence $(f_n(x))_{n \in \mathbb N}$ is Cauchy in $E$. Thus, we can easily conclude that each section $A(x)$, where $x\in X$, is precompact. So, we only need to show that the algebraic difference $A-A$ satisfies the condition (B).  

Let $\varepsilon>0$ be fixed. Because the algebraic difference $A-A$ is equinormed there is a finite set $Y:=\{t_1,\ldots,t_m\} \subseteq X$ such that $\norm{f-g}_\infty \leq \frac{1}{16}\varepsilon + \norm{f-g}_Y$ for any $f,g \in A$. From the first part of the proof we know that each section $A(t_i)$, where $i=1,\ldots,m$, is precompact. Therefore, precompact is also the union $\bigcup_{i=1}^m A(t_i)$. Hence, there exists a finite collection of points $e_1,\ldots,e_n \in E$ such that $\bigcup_{i=1}^m A(t_i) \subseteq \bigcup_{i=1}^n B_E(e_i,\frac{1}{32}\varepsilon)$. By $\Psi$ denote the set of all functions $\psi\colon \{1,\ldots,m\}\to \{1,\ldots,n\}$, and for each $\psi \in \Psi$ let $A_\psi := \dset[\big]{f\in A}{\text{$f(t_i) \in B_E(e_{\psi(i)},\tfrac{1}{32}\varepsilon)$ for every $i=1,\ldots,m$}}$. Moreover, by $\Omega$ denote the set of all those $\psi\in\Psi$ for which $A_\psi\neq\emptyset$. It is obvious that for each $f\in A$ there exists $\psi_f\in \Omega$ such that $f\in A_{\psi_f}$. Let, additionally, for each $\psi\in\Omega$ a function $h_\psi\in A_\psi$ be fixed. As the set $\Omega$ is finite, there are only finitely many functions $h_\psi$. For simplicity, let us denote them as $h_1,\ldots,h_N$. Furthermore, let us define the function $G \colon X \to \mathbb R^{N^2}$ by the formula
\begin{align*}
& G(x):=\bigl(\norm{(h_1-h_1)(x)}_E, \norm{(h_1-h_2)(x)}_E, \ldots, \norm{(h_1-h_N)(x)}_E, \norm{(h_2-h_1)(x)}_E,\\
&\hspace{4.5cm} \ldots, \norm{(h_2-h_N)(x)}_E, \ldots, \norm{(h_N-h_1)(x)}_E,\ldots, \norm{(h_N-h_N)(x)}_E\bigr).
\end{align*}
In view of Remark~\ref{rem:boundedness} the set $A$ is bounded in $B(X,E)$. Let $M:=\sup_{f \in A}\norm{f}_{\infty}$. Then, $G(X)\subseteq [0,2M]^{N^2}$. Let $J_1,\ldots,J_k$ be a family of non-empty intervals of $\mathbb R$ covering $[0,2M]$ of diameters at most $\frac{1}{2}\varepsilon$. Moreover, let
\begin{align*}
& U_\lambda:=G^{-1}\bigl(J_{\lambda_{1,1}}\times J_{\lambda_{1,2}} \times \cdots \times J_{\lambda_{1,N}}\times J_{\lambda_{2,1}} \times \cdots \times J_{\lambda_{2,N}} \times \cdots \times J_{\lambda_{N,1}} \times \cdots \times J_{\lambda_{N,N}}\bigr), 
\end{align*}
where $\lambda:=(\lambda_{1,1}, \lambda_{1,2},\ldots, \lambda_{1,N}, \lambda_{2,1},\ldots, \lambda_{2,N}, \ldots, \lambda_{N,1}, \ldots, \lambda_{N,N}) \in \{1,\ldots,k\}^{N^2}$. The sets $U_\lambda$ form a finite cover of $X$. Let $\Lambda$ be the set of all those indices $\lambda$ for which $U_\lambda \neq \emptyset$.

We will show that for the sets $U_\lambda$, where $\lambda \in \Lambda$, the condition (B) holds. Let us fix $\lambda \in \Lambda$, mappings $f,g \in A$ and points $x,y \in U_\lambda$. Furthermore, let $h_f$ and $h_g$ be those maps among $h_1,\ldots,h_N$ for which we have $h_f \in A_{\psi_f}$ and $h_g \in A_{\psi_g}$. (Note that $h_f$ and $h_g$ may coincide.) Further, instead of writing $\lambda_{p,q}$ with the appropriate numbers $p,q \in \{1,\ldots,N\}$ corresponding to $h_f$ and $h_g$, we will simply write $\lambda_{f,g}$. By the definition of the sets $A_\psi$, we have $f(t_i), h_f(t_i) \in B_E(e_{{\psi_f}(i)},\frac{1}{32}\varepsilon)$ and $g(t_i), h_g(t_i) \in B_E(e_{{\psi_g}(i)},\frac{1}{32}\varepsilon)$ for $i=1,\ldots,m$. This implies that $\norm{f-h_f}_Y\leq \frac{1}{16}\varepsilon$ and $\norm{g-h_g}_Y\leq \frac{1}{16}\varepsilon$. Consequently, $\norm{f-h_f}_\infty\leq \frac{1}{16}\varepsilon + \norm{f-h_f}_Y \leq \frac{1}{8}\varepsilon$ and $\norm{g-h_g}_\infty \leq \frac{1}{8}\varepsilon$. Likewise, by the definition of the sets $U_\lambda$, we have $\norm{(h_f-h_g)(x)}_E, \norm{(h_f-h_g)(y)}_E \in J_{\lambda_{f,g}}$. And so, $\abs[\big]{\norm{(h_f-h_g)(x)}_E - \norm{(h_f-h_g)(y)}_E}\leq \frac{1}{2}\varepsilon$. Note that
\[
 \norm{(f-g)(x)}_E \leq  \norm{(f-h_f)(x)}_E + \norm{(h_f-h_g)(x)}_E + \norm{(h_g-g)(x)}_E
\]
and
\[
 \norm{(h_f-h_g)(y)}_E \leq \norm{(h_f-f)(y)}_E + \norm{(f-g)(y)}_E + \norm{(g-h_g)(y)}_E.
\]
Thus,
\begin{align*}
& \norm{(f-g)(x)}_E - \norm{(f-g)(y)}_E\\
& \quad \leq  \norm{(f-h_f)(x)}_E + \norm{(h_f-h_g)(x)}_E + \norm{(h_g-g)(x)}_E\\
&\qquad + \norm{(h_f-f)(y)}_E  - \norm{(h_f-h_g)(y)}_E + \norm{(g-h_g)(y)}_E\\
& \quad \leq 2\norm{f-h_f}_\infty + 2\norm{g-h_g}_\infty + \norm{(h_f-h_g)(x)}_E - \norm{(h_f-h_g)(y)}_E\\
& \quad \leq 2\cdot \tfrac{1}{8}\varepsilon + 2\cdot \tfrac{1}{8}\varepsilon + \tfrac{1}{2}\varepsilon = \varepsilon.
\end{align*}
Interchanging the roles of $x,y$, we get $\norm{(f-g)(y)}_E - \norm{(f-g)(x)}_E \leq\varepsilon$. This shows that the set $A-A$ satisfies the condition (B).
\end{proof}

\subsection{Discussion of the condition (B)}\label{sec:discussion_of_B}
\label{sec:discussion}

In this section we are going to study the condition (B) little bit more. There are some questions that can (and should) be raised. One of the most natural ones is whether in the statement of Theorem~\ref{thm:compactness_BXE} the phrase ``the set $A-A$ satisfies the condition (B)'' can be replaced by ``the set $A$ satisfies the condition (B)''. It may come as a surprise that the answer is negative. Moreover, it cannot be done in any infinite-dimensional normed space $E$.

\begin{proposition}\label{prop:B:zero:one}
Let $X$ be the set consisting of all zero-one sequences with only finitely many non-zero terms. Moreover, let $E$ be an infinite-dimensional normed space. Then, there exists a non-empty set $A\subseteq B(X,E)$ with the following properties\textup:
\begin{enumerate}[label=\textup{(\alph*)}]
 \item $A$ satisfies the condition \textup{(B)},
 \item $A-A$ does not satisfy the condition \textup{(B)},
 \item for each $x \in X$ the section $A(x)$ is compact,
 \item $A(X)$ is not precompact,
 \item $A$ is not precompact.
\end{enumerate}
\end{proposition}

\begin{proof}
By the well-known Riesz lemma (see e.g.~\cite{kreyszig}*{Theorem~2.5-4}), there exists a sequence $(e_n)_{n \geq 0}$ of unit vectors in $E$ such that $\norm{e_i-e_j}_E \geq \frac{1}{2}$ for all distinct $i,j\in \N\cup\{0\}$. For each $n \in \mathbb N$ let us define a bounded mapping $f_n \colon X \to E$ by
\begin{equation}
f_n(x) = \begin{cases}
  e_0 & \text{if $\xi_n = 0$,}\\
	e_n & \text{if $\xi_n = 1$;}
\end{cases}
\end{equation}
here $x:=(\xi_n)_{n \in \mathbb N}$. Let also $A:=\dset{f_n}{n\in \mathbb N}$. 

Note that the set $A$ satisfies the condition (B) with the cover of $X$ consisting of a single set $U_1:=X$, because for every $x,y\in X$ and $n \in \mathbb N$ we have $\norm{f_n(x)}_E = \norm{f_n(y)}_E = 1$. However, the algebraic difference $A-A$ does not satisfy condition (B). To see this let $U_1,\ldots,U_N$ be any finite cover of $X$ and let $k$ be a positive integer such that $2^{k-1} > N$. Since there are exactly $2^{k-1}$ different zero-one sequences of length $k-1$, there exists at least one set $U_i$ which contains at least two distinct elements $x=(\xi_n)_{n \in \mathbb N}$ and $y=(\eta_n)_{n \in \mathbb N}$ such that $\xi_n = \eta_n = 0$ for $n\geq k$. As $x\neq y$, an index $l<k$ exists such that $\xi_l\neq \eta_l$. Of course, we may assume that $\xi_l=1$ and $\eta_l=0$. So, $f_k(x) = f_k(y) = f_l(y)= e_0$ and $f_l(x)=e_l$. Therefore, we have
\[
\abs[\big]{\norm{(f_k-f_l)(x)}_E - \norm{(f_k-f_l)(y)}_E} = \norm{e_0-e_l}_E \geq \tfrac{1}{2}.
\]
This shows that the condition (B) for the set $A-A$ is not satisfied.

Now, let us fix $x=(\xi_n)_{n \in \mathbb N} \in X$ and let $m \in \mathbb N$ be such that $\xi_n=0$ for $n \geq m$. Then, $A(x)\subseteq \{e_0,e_1,\ldots,e_{m-1}\}$, meaning that the section $A(x)$ is compact. On the other hand, $A(X)=\dset{e_n}{n \in \mathbb N\cup\{0\}}$. This implies that $A(X)$ is not precompact, as the sequence $(e_n)_{n \in \mathbb N}$ does not contain a Cauchy subsequence. 

The fact that the set $A$ is not precompact follows from Theorem~\ref{thm:compactness_BXE}.
\end{proof}

When dealing with compactness criteria for real-valued functions, a one-dimensional version of the following condition often can be encountered:
\begin{enumerate}
 \item[(DS)] for every $\varepsilon>0$ there exist finitely many non-empty subsets $U_1,\ldots,U_N$ of $X$ whose union is $X$ such that for every $i \in \{1,\ldots,N\}$ and every pair $x, y \in U_i$ for all $f \in A$ we have $\norm{f(x) - f(y)}_E \leq \varepsilon$ 
\end{enumerate}
(cf.~\cite{alexiewicz1969analiza}*{pp. 370–371} or~\cite{DS}*{Section IV.4}). Clearly, it is stronger than the condition (B). Thus, another natural question concerning Theorem~\ref{thm:compactness_BXE} arises. Does this result still hold, if in its statement we replace the condition (B) with (DS)? Or, maybe we should impose the condition (DS) on the set $A$ rather than on its algebraic difference $A-A$. The answers to both these questions are negative, as shown by the following example.

\begin{example}\label{ex:B:comp}
Let $X$ be the unit sphere of an infinite-dimensional normed space $E$. Consider the mappings $f,g \colon X \to E$ given by $f(x)=x$ and $g(x)=2x$. Moreover, let $A:=\{f,g\}$. It is easy to see that both the sets $A$ and $A-A$ satisfy the condition (B) with $U_1:=X$, as $\norm{f(x)}_E- \norm{f(y)}_E=0$, $\norm{g(x)}_E-\norm{g(y)}_E=0$ and $\norm{(f-g)(x)}_E-\norm{(f-g)(y)}_E=0$ for any $x,y \in X$. However, if either $A$ or $A-A$ satisfied the condition (DS), then for every $\varepsilon>0$ it would be possible to cover the unit sphere of an infinite-dimensional normed space with finitely many sets of diameter not greater than $\varepsilon$, which is absurd. Thus, neither the set $A$ nor the algebraic difference $A-A$ satisfy the condition (DS). Note also that for each $x \in X$ the section $A(x)$ is precompact (it is even compact), but the set $A(X)$ is not.
\end{example}

Let us stop for a moment and look at the above example once more. The construction of the set $A$ with appropriate properties was possible, because the normed space $E$ was infinite-dimensional. In consequence, the set $A(X)$ was not precompact. However, if we assume that the subset $A(X)$ of $E$ is precompact, the situation changes significantly. 

\begin{proposition}\label{prop:DS_equivalent_B}
For a non-empty subset $A$ of $B(X,E)$ the following conditions are equivalent\textup:
\begin{enumerate}[label=\textup{(\roman*)}]
 \item\label{it:DS_B_i}  $A$ satisfies the condition \textup{(DS)} and for every $x \in X$ the section $A(x)$ is precompact,
 \item\label{it:DS_B_ii} $A-A$  satisfies the condition \textup{(B)} and the set $A(X)$ is precompact.
\end{enumerate}
\end{proposition}

\begin{proof}
We begin with the proof of the implication~$\ref{it:DS_B_i}\Rightarrow\ref{it:DS_B_ii}$. Note that if the set $A$ satisfies the condition (DS), then so does $A-A$. And, moreover, for any $x,y\in X$ and $f,g \in A$ we have $\abs[\big]{\norm{(f-g)(x)}_E-\norm{(f-g)(y)}_E}\leq \norm{(f-g)(x)-(f-g)(y)}_E$. Now, we need only to show that the set $A(X)$ is precompact. Fix $\varepsilon>0$. By the condition (DS) the set $X$ can be covered with a finite collection of non-empty sets $U_1,\ldots,U_n$ such that $\sup_{f \in A}\norm{f(\xi)-f(\eta)}_E \leq \frac{1}{2}\varepsilon$ for every $i\in\{1,\ldots,n\}$ and every pair $\xi,\eta \in U_i$. In each set $U_i$ let us fix a point $t_i$. Then, for any $x \in U_i$ and $f \in A$ we have $f(x)=f(t_i)+f(x)-f(t_i) \in \{f(t_i)\}+\ball_E(0,\frac{1}{2}\varepsilon)$. Thus, $A(U_i):=\bigcup_{x\in U_i}A(x) \subseteq A(t_i) + \ball_E(0,\frac{1}{2}\varepsilon)$. But $A(t_i)$ is totally bounded by assumption, and so it can be covered with a finite family of balls $\ball_E(y_1^i,\frac{1}{2}\varepsilon)$, \ldots, $\ball_E(y_{m_i}^i,\frac{1}{2}\varepsilon)$, where $y_j^i \in A(t_i)$. This means that $A(X)\subseteq \bigcup_{i=1}^n A(U_i) \subseteq \bigcup_{i=1}^n \bigcup_{j=1}^{m_i} \ball_E(y_j^i,\varepsilon)$, and shows that $A(X)$ is totally bounded/precompact.

Now, we prove the implication~$\ref{it:DS_B_ii}\Rightarrow\ref{it:DS_B_i}$. The reasoning is similar to the one we used to establish Lemma~\ref{lem:7}. Let us assume that the algebraic difference $A-A$ satisfies the condition (B). Then, given $\varepsilon>0$ there exists a finite cover $W_1,\ldots,W_n$ of $X$ such that for every $i \in \{1,\ldots,n\}$ and all $x,y \in W_i$ and $f,g \in A$ we have $\abs[\big]{\norm{(f-g)(x)}_E - \norm{(f-g)(y)}_E} \leq \frac{1}{8}\varepsilon$. For each $i \in \{1,\ldots,n\}$ let us also fix $w_i\in W_i$. Furthermore, let $e_1,\ldots,e_m \in E$ be a finite $\frac{1}{8}\varepsilon$-net for $A(X)$.

By $\Psi$ denote the set of all functions $\psi\colon \{1,\ldots,n\}\to \{1,\ldots,m\}$. For each $\psi \in \Psi$ let $A_\psi: = \dset{f\in A}{\text{$f(w_i) \in \ball_E(e_{\psi(i)},\tfrac{1}{8}\varepsilon)$ for every $i=1,\ldots,n$}}$.
 Moreover, by $\Omega$ denote the set of all those $\psi\in\Psi$ for which $A_\psi\neq\emptyset$. 
 For each $\psi\in\Omega$ let us chose a function $g_\psi\in A_\psi$. Thus, we obtain a finite family $g_1,\ldots,g_k$ of functions, where $k$ is the cardinality of $\Omega$.  
  It is obvious that for each $f\in A$ there exist $\psi_f\in \Omega$ and a corresponding function $g_l \in \{g_1,\ldots,g_k\}$ such that $f,g_l\in A_{\psi_f}$.
 
Define the mapping $G \colon X \to E^k$ by the formula $G(x)=(g_1(x),\ldots,g_k(x))$, and let $V_\lambda:=G^{-1}(\prod_{i=1}^k \ball_E(e_{\lambda_i},\frac{1}{8}\varepsilon))$, where $\lambda:=(\lambda_1,\ldots,\lambda_k) \in \{1,\ldots,m\}^k$. Set $U_{i,\lambda}:=W_i \cap V_\lambda$. By $I$ let us denote the set of those pairs $(i,\lambda)$, where $i \in \{1,\ldots,n\}$ and $\lambda \in \{1,\ldots,m\}^k$, for which the sets $U_{i,\lambda}$ are non-empty. Clearly, the finite family $\{U_{i,\lambda}\}_{(i,\lambda) \in I}$ is a covering of $X$. 

Let us now fix an index $(i,\lambda) \in I$, take any points $x,y\in U_{i,\lambda}$ and any function $f \in A$, together with the corresponding function $g_l \in \{g_1,\ldots,g_k\}$, such that $f,g_l\in A_{\psi_f}$. Then, by condition (B) and the fact that $U_{i,\lambda}\subseteq W_i$, we can see that 
\[
\norm{(f-g_l)(x)}_E\leq \norm{(f-g_l)(w_i)}_E + \tfrac{1}{8}\varepsilon.
\]
Similarly,
\[
\norm{(f-g_l)(y)}_E\leq \norm{(f-g_l)(w_i)}_E + \tfrac{1}{8}\varepsilon.
\]
As $f,g_l \in A_{\psi_f}$, we have $\norm{(f-g_l)(w_i)}_E \leq \frac{1}{4}\varepsilon$. Using the fact that the set $U_{i,\lambda}$ is included in $V_\lambda$, we obtain $\norm{g_l(x)-g_l(y)}_E \leq \frac{1}{4}\varepsilon$. Therefore, $\norm{f(x)-f(y)}_E \leq \norm{(f-g_l)(x)}_E + \norm{g_l(x)-g_l(y)}_E + \norm{(f-g_l)(y)}_E \leq \tfrac{3}{8}\varepsilon + \tfrac{1}{4}\varepsilon + \tfrac{3}{8}\varepsilon = \varepsilon$.

Since a subset of a precompact set is also precompact, and $A(x)\subseteq A(X)$ for every $x \in X$, this ends the proof.
\end{proof}

As bounded sets in finite-dimensional Banach spaces are precompact, from Theorem~\ref{thm:compactness_BXE} and Proposition~\ref{prop:DS_equivalent_B} (cf. also Remark~\ref{rem:relatively_compact}) we immediately get the following result. For $n=1$ and $\mathbb K:=\mathbb R$ it can be found (in a slightly different but equivalent form) on page~260 in~\cite{DS}. 

\begin{theorem}\label{thm:DS}
Let $n \in \mathbb N$ and $\mathbb K \in \{\mathbb R, \mathbb C\}$. A non-empty set $A \subseteq B(X,\mathbb K^n)$ is relatively compact if and only if it is bounded and satisfies the condition \textup{(DS)}.
\end{theorem}

\section{Compactness in $C(X,E)$}
\label{sec:compactness_in_CXE}

In this section, as corollaries to our compactness results established previously, we will state compactness criteria for the space of bounded and continuous maps. The most interesting is the first one, as it is probably not (widely) known. The other two (in the exact or similar version) are already known in the literature (see, e.g.,~\cite{alexiewicz1969analiza}*{pp.~370--371}, \cite{BBK}*{Corollary~1.2.9}, \cite{DS}*{Section~IV.6} or~\cite{kreyszig}*{Theorem~8.7-4}). We have decided to state them, because of two reasons. Firstly, we will need them when dealing with compactness in Lipschitz spaces. And, secondly, because, in contrast to the approach presented in~\cite{DS}, our method allows to say more about the regularity of the sets $U_1,\ldots,U_n$ appearing in the conditions (B) and (DS) (see Remarks~\ref{rem:open_closed} and~\ref{rem:open_closed2}).  

As an application of Theorem~\ref{thm:compactness_BXE} to a closed subspace of $B(X,E)$ consisting of continuous mappings we get the following result.

\begin{theorem}\label{thm:compactness_in_CXE_1}
Let $X$ be a topological space. Then, a non-empty subset $A$ of $C(X,E)$ is precompact if and only if the algebraic difference $A-A$ satisfies the condition \textup{(B)} and for every $x \in X$ the sections $A(x)$ are precompact.  
\end{theorem}
 
\begin{remark}\label{rem:open_closed}
A closer look at the proof of Theorem~\ref{thm:compactness_BXE} allows us to say more about the regularity of the sets $U_1,\ldots,U_N$ appearing in the condition (B) in Theorem~\ref{thm:compactness_in_CXE_1}. If $A\subseteq C(X,E)$, then the function $G$ defined in the proof of Lemma~\ref{lem:7} is continuous. Thus, taking the open/closed covering $J_1,\ldots,J_k$ of the set $[0,2M]^{N^2}$, we may conclude that all the set $U_1,\ldots,U_N$ are open (closed) in $X$.
\end{remark}

We already know that in some situations both conditions (B) and (DS) can be used to characterize precompact subsets of $B(X,E)$. Now, we prove a result similar to Proposition~\ref{prop:DS_equivalent_B} in the case when $A$ is a non-empty subset of the space $C(X,E)$.   

\begin{corollary}\label{rem:DS_B_3}
Let $X$ be a compact Hausdorff topological space. Moreover, assume that $A$ is a non-empty subset of $C(X,E)$ such that each section $A(x)$, where $x \in X$, is precompact. Then, $A$ satisfies the condition \textup{(DS)} if and only if  $A-A$  satisfies the condition \textup{(B)}.
\end{corollary} 

\begin{proof}
In view of Proposition~\ref{prop:DS_equivalent_B} it suffices to show that if the algebraic difference $A-A$ satisfies the condition (B), then the set $A(X)$ is precompact. So let us fix $\varepsilon>0$. Then, by Theorem~\ref{thm:compactness_in_CXE_1} the set $A$ is precompact, and hence it has a finite $\varepsilon$-net $g_1,\ldots,g_n \in A$. As the functions $g_i$ are continuous and the topological space $X$ is compact, the images $g_i(X)$ as well as their union $\bigcup_{i=1}^n g_i(X)$ are compact. For every $f \in A$ and $x \in X$ we thus have $f(x)=f(x)-g_j(x) + g_j(x) \in \ball_E(0,\varepsilon)+g_j(X)$, where the function $g_j$ is an element of the $\varepsilon$-net $g_1,\ldots,g_n$ chosen so that $\norm{f-g_j}_\infty \leq \varepsilon$. Therefore, $A(X)\subseteq \ball_E(0,\varepsilon)+\bigcup_{i=1}^n g_i(X)$. This proves that the set $A(X)$ is precompact (cf. the first part of the proof of Proposition~\ref{prop:DS_equivalent_B}).
\end{proof}

\begin{remark}
It is worth underlining that using a similar approach to the one we used in the proof of Proposition~\ref{prop:DS_equivalent_B}, we can avoid applying Theorem~\ref{thm:compactness_in_CXE_1} in the proof of Corollary~\ref{rem:DS_B_3}. We decided, however, to proceed along this less elegant route to simplify the reasoning and not utilizing the same arguments over and over again.     
\end{remark}

Combining Proposition~\ref{prop:DS_equivalent_B}, Theorem~\ref{thm:compactness_in_CXE_1} and Corollary~\ref{rem:DS_B_3}, yields yet another one compactness criterion in $C(X,E)$. 

\begin{theorem}\label{thm:compactness_in_CXE_2}
Let $n \in \mathbb N$ and $\mathbb K \in \{\mathbb R, \mathbb C\}$. Moreover, assume that either
\begin{enumerate}[label=\textup{(\alph*)}]
 \item\label{it:CXE_2_i} $X$ is an arbitrary topological space and $E=\mathbb K^n$, or
 \item $X$ is a compact Hausdorff topological space and $E$ is an arbitrary normed space.
\end{enumerate}
Then, a non-empty subset $A$ of $C(X,E)$ is precompact if and only if it satisfies the condition \textup{(DS)} and the sections $A(x)$ for $x \in X$ are precompact.
\end{theorem}

\begin{remark}\label{rem:17}
Let $n \in \mathbb N$ and $\mathbb K \in \{\mathbb R, \mathbb C\}$. Note that when  $E=\mathbb K^n$, the set $A(x)$ for $x\in X$ is precompact if and only if it is bounded. Therefore, in this case, in Theorem~\ref{thm:compactness_in_CXE_2} instead of the phrase ``the sections $A(x)$ for $x \in X$ are precompact'' we can equivalently write ``the sections $A(x)$ for $x \in X$ are bounded'', or even ``$A$ is a bounded subset of $C(X,\mathbb K^n)$''. However, when the normed space $E$ is infinite-dimensional, the precompactness of the sections $A(x)$ cannot be replaced with their boundedness, or the boundedness of the whole set $A$. To see this take a separable infinite-dimensional Banach space $E$ and set $X:=\ball_{E^\ast}(0,1)$, where $E^\ast$ is the dual of $E$. From classical facts from functional analysis we know that $X$ is a metrizable compact Hausdorff topological space when endowed with the weak$^\ast$ topology (see~\cite{Conway}*{Theorem~V.3.1 and Theorem~V.5.1}). Given a point $y \in X$ consider the constant mapping $g_y \colon X \to (E^\ast,\norm{\cdot}_{E^\ast})$ given by $g_y(x)=y$ for $x \in X$. Also, let $A:=\dset{g_y}{y \in X}$. It is straightforward to check that the set $A$ satisfies the condition (DS) (with a family consisting of a single open set $U_1:=X$) and is bounded in $C(X,E^{\ast})$. However, $A(x)=\ball_{E^\ast}(0,1)$ for any $x \in X$. So, no section is precompact in the norm of $E^\ast$.
\end{remark}

\begin{remark}\label{rem:open_closed2}
A remark similar to Remark~\ref{rem:open_closed} is also true in the case of Theorem~\ref{thm:compactness_in_CXE_2}. In other words, for a set $A\subseteq C(X,E)$, in the condition (DS) we may require all the sets $U_1,\ldots,U_N$ to be either open or closed in $X$.
\end{remark}

As a corollary to Theorem~\ref{thm:compactness_in_CXE_2} we obtain the classical Arz\`ela--Ascoli compactness criterion. It uses the notion of an equicontinuous family of functions. Although, this notion is very well-known, just for completeness, let us recall it. A non-empty set $A \subseteq C(X,E)$ is \emph{equicontinuous}, if for every $\varepsilon>0$ there exists $\delta>0$ such that for every $x,y \in X$ with $d(x,y)\leq \delta$ and every $f \in A$ we have $\norm{f(x)-f(y)}_E \leq \varepsilon$.

\begin{theorem}\label{thm:compactness_in_CXE_3}
Let $(X,d)$ be a compact metric space. Then, a non-empty subset $A$ of $C(X,E)$ is precompact if and only if it is equicontinuous and for every $x \in X$ the sections $A(x)$ are precompact.
\end{theorem}

\begin{proof}
Assume that $A\subseteq C(X,E)$ is equicontinous. Let us fix $\varepsilon>0$ and choose $\delta>0$ according to the definition of equicontinuity. Since the family $\bigl\{ B_X(x,\frac{1}{2}\delta)\bigr\}_{x \in X}$ is an open cover of the compact metric space $X$, there exist $x_1,\ldots,x_N \in X$ such that $X=\bigcup_{i=1}^N B_X(x_i,\frac{1}{2}\delta)$. Let $U_i:=B_X(x_i,\frac{1}{2}\delta)$ for $i=1,\ldots,N$. It is now clear that if $x,y \in U_i$, then $d(x,y)\leq \delta$, and hence $\sup_{f \in A}\norm{f(x)-f(y)}_E \leq \varepsilon$. This means that $A$ satisfies the condition (DS). As the sections $A(x)$ for $x \in X$ are precompact, it suffices to apply Theorem~\ref{thm:compactness_in_CXE_2} to conclude that $A$ is precompact. 

Let us move to the second part of the proof. This time we assume that $A$ is precompact. Then, in view of Theorem~\ref{thm:compactness_in_CXE_2} and Remark~\ref{rem:open_closed2}, it satisfies the condition (DS) with open sets. Also, for each $x \in X$ the set $A(x)$ is precompact. Our aim is to show that $A$ is equicontinuous. So, let us fix $\varepsilon>0$, and let $U_1,\ldots,U_N$ be the open cover of $X$ appearing in the condition (DS). Given a set $U_i$ and $\xi_i \in U_i$, let $r_{\xi_i}>0$ be such that $B_X(\xi_i,r_{\xi_i})\subseteq U_i$. Then, the family $\bigcup_{i=1}^N \bigl\{B_X(\xi_i,\frac{1}{2}r_{\xi_i})\bigr\}_{\xi_i \in U_i}$ is an open cover of the compact metric space $X$. Hence, it contains a finite subcover. Let us denote the elements of this subcover by $B_X(x_1,\frac{1}{2}r_1),\ldots,B_X(x_m,\frac{1}{2}r_m)$. Moreover, let $\delta:=\frac{1}{2}\min_{1\leq k \leq m} r_k$. Now, choose arbitrary $x,y \in X$ such that $d(x,y)\leq \delta$, and fix $f \in A$. Then, $x \in B_X(x_j,\frac{1}{2}r_j)\subseteq U_l$ for some $j \in \{1,\ldots,m\}$ and $l \in \{1,\ldots,N\}$. Consequently, $y \in B_X(x_j,r_j) \subseteq U_l$. Now, it is enough to use the condition (DS) to conclude that $\norm{f(x)-f(y)}_E \leq \varepsilon$. In other words, $A$ is equicontinuous. 
\end{proof}

\section{Compactness in Lipschitz spaces}
\label{sec:compact_in_lip}

The goal of this section is to provide a full characterization of precompact subsets of both the Lipschitz and little Lipschitz spaces. In the first part of this section we will work with functions defined on an arbitrary metric space. In the second one, we will move to the case of compact domains.

\begin{notation}
Throughout this section for a given metric space $(X,d)$ we will write $\tilde{X}$ for the metric space $(X\times X) \setminus \dset{(x,x)}{x \in X}$. We endow $\tilde{X}$ (and any other subset of $X\times X$) with the maximum metric $d_\infty$ inherited from $X \times X$. For completeness, let us add that for $(x_1,x_2),(y_1,y_2) \in X\times X$ the \emph{maximum metric} is defined by the formula $d_\infty((x_1,x_2),(y_1,y_2)) = \max\{ d(x_1,y_1), d(x_2,y_2) \}$.
\end{notation}

\subsection{Comparison function}
A key ingredient when defining classes of Lipschitz continuous mappings is the so-called \emph{comparison function}, that is, a non-zero function $\varphi\colon [0,+\infty)\to[0,+\infty)$ that is right-continuous at $0$, concave and such that $\varphi(0)=0$. Typical examples of such functions are the power functions $\varphi(t)=t^\alpha$ for $\alpha \in (0,1]$. They are especially important, as they give rise to the classical classes of Lipschitz/H\"older continuous mappings (for the appropriate definitions see the next subsection). Another example of a comparison function is $\varphi(t)=\ln(1+t)$. It may be interesting, because near zero it behaves like the identity function, but when $t$ tends to infinity, it increases slower than any of the power functions defined above. In the sequel we will use several basic properties of such functions; we gathered them in the following lemma. We skip its proof, because it is straightforward and the arguments can be found scattered around various books on real functions or the internet (especially, on the Mathematics Stack Exchange forum -- see, for example, question 2757353).

\begin{lemma}\label{lem:comparision_function}
Let $\varphi\colon [0,+\infty)\to[0,+\infty)$ be a comparison function. Then, 
\begin{enumerate}[label=\textup{(\alph*)}]
 \item\label{lem:comparision_function_a} $\varphi$ is continuous on $[0,+\infty)$,
 \item\label{lem:comparision_function_b} $\varphi(t)/t$ is non-increasing on $(0,+\infty)$,
 \item\label{lem:comparision_function_c} $\varphi$ is sub-additive, that is, $\varphi(t+s)\leq \varphi(t)+\varphi(s)$ for $t,s\in [0,+\infty)$,
 \item\label{lem:comparision_function_d} $\varphi$ is non-decreasing on $[0,+\infty)$,
 \item\label{lem:comparision_function_e} $\varphi(t)>0$ for every $t \in (0,+\infty)$,

 \item\label{lem:comparision_function_f} the limit $\lim_{t \to 0^+} \varphi(t)/t$ exists, and is either a positive number or $+\infty$.
\end{enumerate} 
\end{lemma}

A word of caution is in order here. The notion of a comparison function is ambiguous and its definition varies throughout the literature. Different classes of comparison functions are defined by means of various combinations of the properties~\ref{lem:comparision_function_a}--\ref{lem:comparision_function_f} and many others. (See~\cite{AKORR}*{Section~2} for a detailed discussion of several types of comparison functions.) Our assumption of concavity may thus seem too strong. And, in fact, it is. However, we decided to use it because of two reasons. It makes all the unnecessary technicalities disappear. And, more importantly, the comparison functions corresponding to the Lipschitz/H\"older classes are concave. In other words, our results cover the most interesting cases.

\subsection{Lipschitz spaces}
\label{sec:52}
Fix a comparison function  $\varphi$. Moreover, let $(X,d)$ be a metric space and let $(E,\norm{\cdot}_E)$ be a normed space. (Note that we \emph{do not} assume that $X$ is bounded or compact.) By $\Lip_\varphi(X,E)$ we denote the class of all $\varphi$-Lipschitz continuous mappings, that is, mappings $f \colon X \to E$ which satisfy the condition
\[
 \abs{f}_{\varphi}:=\sup_{\substack{x,y \in X\\x\neq y}}\frac{\norm{f(x)-f(y)}_E}{\varphi(d(x,y))}<+\infty.
\]
Now, let us fix a point $x_*$ in $X$; such a point is called a \emph{base} (or \emph{distinguished}) \emph{point} of the space $X$. It can be easily checked that $\Lip_\varphi(X,E)$ is a normed space, when endowed with the norm $\norm{f}_\varphi:=\norm{f(x_*)}_E+\abs{f}_{\varphi}$. It should be also evident that for different selections of base points the corresponding norms are equivalent. Furthermore, if $E$ is a Banach space, then  the space $\Lip_\varphi(X,E)$ is complete. 

If $\varphi(t)=t^\alpha$ for $\alpha \in (0,1)$, the space $\Lip_\varphi(X,E)$ coincides with the space of H\"older continuous maps with exponent $\alpha$, and is usually denoted by $\Lip_\alpha(X,E)$. Similarly, if $\varphi(t)=t$, then $\Lip_\varphi(X,E)$ is just the space of Lipschitz mappings $\Lip(X,E)$.

At the space $\Lip_\varphi(X,E)$ we may also look from a slightly different perspective. With the help of Lemma~\ref{lem:comparision_function} it is easy to see that if $\varphi$ is a comparison function and $d$ is a metric on $X$, so is $d_\varphi:=\varphi \circ d$. Therefore, if by $Y$ we denote the metric space $(X,d_\varphi)$, then the space $\Lip(Y,E)$ is well-defined. Moreover, in such a case, $\Lip_{\varphi}(X,E)$ and $\Lip(Y,E)$ are linearly isomorphic. (If we chose the same base point in $X$ and $Y$, then they are even isometrically isomorphic.) So, in the sequel, instead of working in spaces $\Lip_{\varphi}(X,E)$, we could focus on the space of Lipschitz continuous mappings entirely. However, we will not do that. The reason behind our decision is practical. Often, in applications, the more explicit the statements of the results one applies are, the better. But, generally, it is just a matter of taste, which approach to choose. 

It is worth noting that because we do not require the space $X$ to have finite diameter, $\Lip_{\varphi}(X,E)$ may contain unbounded maps. In the literature, the space $\BLip_\varphi(X,E)$ of bounded $\varphi$-Lipschitz continuous mappings is also  considered (cf.~\cite{cobzas_book}*{Chapter~8} or~\cite{weaver}*{Chapter~2}). It is endowed with the norm $\vnorm{f}_\varphi := \norm{f}_\infty + \abs{f}_\varphi$. For arbitrary domains, $\BLip_\varphi(X,E)$ is just a linear subspace of $\Lip_\varphi(X,E)$. Moreover, in general, the norm $\vnorm{\cdot}_\varphi$ on $\BLip_\varphi(X,E)$ is strictly stronger than $\norm{\cdot}_\varphi$ inherited from $\Lip_\varphi(X,E)$. However, if $X$ has finite diameter (or, in particular, is compact), then $\BLip_\varphi(X,E)$ is linearly isomorphic to $\Lip_\varphi(X,E)$. 

Another space of Lipschitz continuous mappings, which plays a prominent role, is the space $\Lip_0(X,E)$. It consists of those Lipschitz maps that vanish at the base point, and is endowed with the norm $f \mapsto \abs{f}_1$; here by $\abs{\cdot}_1$ we denote the semi-norm $\abs{\cdot}_\varphi$ corresponding to the comparison function $\varphi(t)=t$. For an arbitrary domain $X$, $\Lip_0(X,E)$ is a closed subspace of $\Lip(X,E)$, with the quotient $\Lip(X,E)/\Lip_0(X,E)$ isometrically isomorphic to $E$. What may come as a surprise, however, is the fact that $\Lip_0(X,E)$ spaces are very closely related to spaces $\BLip(X,E)$. It turns out that for $E=\mathbb R$ every space of bounded Lipschitz continuous functions, as Weaver puts it on page~42 of his monograph~\cite{weaver}, ``\emph{effectively is, in every important sense, a Lip$_0$ space}.''

The main takeaway from the above discussion is that the (pre)compactness criteria we are going to prove are quite general. Not only can they be applied to the space $\Lip_{\varphi}(X,E)$ and  $\Lip_0(X,E)$ with arbitrary metric space $X$, but also to spaces $\BLip_\varphi(X,E)$. 

For more information on various classes of Lipschitz continuous functions and their properties, we refer the Reader to~\cites{cobzas_book, weaver}.

\subsection{De Leeuw's map}\label{sec:deLeeuw}
When dealing with Lipschitz continuous mappings, the de Leeuw's map $\Phi$ comes in handy. To each $f \in \Lip_\varphi(X,E)$, it assigns the mapping $\Phi(f) \colon \tilde X \to E$ given by  
\[
\Phi(f)(x,y) = \frac{f(x)-f(y)}{\varphi(d(x,y))}.
\]
The mapping $\Phi(f)$ is clearly continuous and bounded on $\tilde X$. Moreover, $\Phi(f)=\Phi(g)$ if and only if the difference $f-g$ is a constant function. More on the de Leeuw's map can be found in~\cite{cobzas_book}*{Section 8.3.1} and \cite{weaver}*{Section 2.4}. 

\subsection{Compactness criteria -- preparatory part}
Before we study (pre)compactness in spaces of Lipschitz continuous mappings, for a non-empty set $A\subseteq L_\varphi(X,E)$ let us introduce the following two conditions:
\begin{enumerate}
 \item[(LDS)] for every $\varepsilon>0$ there is a finite cover $U_1, \ldots, U_N$ of $\tilde X$ such that for any $i\in \{1,\ldots, N\}$ we have
 \[
 \sup_{(x,y), (\xi,\eta) \in U_i} \sup_{f \in A} \norm[\bigg]{\frac{f(x) - f(y)}{\varphi(d(x,y))} - \frac{f(\xi) - f(\eta)}{\varphi(d(\xi,\eta))}}_E\leq \varepsilon,
\]

 \item[(L)] for every $\varepsilon>0$ there is a finite cover $U_1,\ldots,U_N$ of $\tilde X$ such that for any $i \in \{1,\ldots,N\}$ we have
 \[
 \sup_{(x,y), (\xi,\eta) \in U_i} \sup_{f \in A}\abs[\Bigg]{\frac{\norm{f(x) - f(y)}_E}{\varphi(d(x,y))} - \frac{\norm{f(\xi) - f(\eta)}_E}{\varphi(d(\xi,\eta))}}\leq \varepsilon.
 \]  
\end{enumerate}

\begin{remark}\label{rem:open_closed_in_L_LDS}
In the above conditions we may also require all the sets $U_1,\ldots,U_N$ to be either open or closed in $\tilde{X}$ (cf.~Remarks~\ref{rem:open_closed} and~\ref{rem:open_closed2}). 
\end{remark}

\subsection{Compactness criteria -- arbitrary domains}
As the title suggests, in this section we will be interested in Lipschitz continuous maps that are defined on an arbitrary metric space. We begin with the space $\Lip_\varphi(X,E)$.

\begin{theorem}\label{thm:compact_in_Lip}
A non-empty subset $A$ of $\Lip_\varphi(X,E)$ is precompact if and only if
 the set $A-A$ satisfies the condition (L) and the sections $A(x)$ for every $x \in X$ are precompact.
\end{theorem}

\begin{proof}
Let $x_\ast$ be the base point of the metric space $X$. Consider the linear operator $T\colon \Lip_\varphi(X,E)\to E\times C(\tilde{X},E)$ given by $T(f)=(f(x_\ast), \Phi(f))$; here $\Phi$ is the de Leeuw's map. It is easy to check that $T$ is an isometric embedding of $\Lip_\varphi(X,E)$ into $E\times C(\tilde{X},E)$, when the target space is endowed with the norm $\norm{(e,g)}:=\norm{e}_E + \norm{g}_\infty$ for $(e,g) \in E\times C(\tilde{X},E)$. Note that the condition (L) for $A-A$ states that the set $\Phi(A)-\Phi(A) \subseteq C(\tilde X,E)$ satisfies the condition (B). Furthermore, precompactness of each section $A(x)$ for $x \in X$ is equivalent with the fact that the sets $A(x_\ast)$ and $\Phi(A)(x,y):=\dset{\Phi(f)(x,y)}{ f\in A}$ for $(x,y) \in \tilde X$ are precompact.

Now, let us assume that the non-empty subset $A$ of $\Lip_\varphi(X,E)$ is precompact. Then, precompact are also the sets $\pi_1(T(A))$ and $\pi_2(T(A))$; here $\pi_1$ and $\pi_2$ are the projections onto the first and second factor of $E\times C(\tilde{X},E)$, respectively.  Note that $\pi_1(T(A))=A(x_\ast)$ and $\pi_2(T(A))=\Phi(A)$. So, by Theorem~\ref{thm:compactness_in_CXE_1} and our preliminary observations, we infer that the condition (L) is satisfied for $A-A$. Also, $A(x)$ is precompact for $x \in X$.

On the other hand, if $A-A$ satisfies the condition (L) and the sections $A(x)$ are precompact for every $x \in X$, then by~Theorem~\ref{thm:compactness_in_CXE_1} the sets $A(x_\ast) \subseteq E$ and $\Phi(A) \subseteq C(\tilde X,E)$ are precompact. Hence, precompact is also their product $A(x_\ast) \times \Phi(A)$. As precompactness is hereditary, this implies the precompactness of $T(A) \subseteq A(x_\ast) \times \Phi(A)$. To end the proof it suffices to note that $A=T^{-1}(T(A))$. 
\end{proof}

In the finite-dimensional case, using Theorem~\ref{thm:compactness_in_CXE_2} instead of Theorem~\ref{thm:compactness_in_CXE_1} and reasoning as in the proof of the above result, we get the following criterion (cf. also Remark~\ref{rem:17}).

\begin{theorem}\label{thm:compact_in_Lip2}
Let $n \in \mathbb N$ and $\mathbb K \in \{\mathbb R, \mathbb C\}$. A non-empty subset $A$ of $\Lip_\varphi(X,\mathbb K^n)$ is precompact if and only if it is bounded and satisfies the condition (LDS).
\end{theorem}

\begin{remark}\label{rem:Lip_0}
It is worth underlining that Theorems~~\ref{thm:compact_in_Lip} and~\ref{thm:compact_in_Lip2} can be also applied to $\Lip_0(X,E)$ with any metric space $X$. This follows from the fact that $\Lip_0(X,E)$ is a closed subspace of $\Lip(X,E)$.
\end{remark}

From Remark~\ref{rem:open_closed_in_L_LDS} we know that in the conditions (LDS) and (L) we may require all the sets $U_1,\ldots,U_N$ to be open in $\tilde{X}$. It would be tempting to replace those sets with open balls in Theorems~\ref{thm:compact_in_Lip} and~\ref{thm:compact_in_Lip2}. Unfortunately, this is impossible, even when $E$ is one-dimensional, as the following example shows. 

\begin{example}\label{ex:balls_open_sets}
Set $X:=[0,2]$ and consider the function $f \colon X \to \mathbb R$ given by $f(x)=1-\abs{x-1}$. It clearly belongs to $\Lip(X,\mathbb R)$. Now, take any finite family of open balls in $\tilde X$ that covers $\tilde X$. Let $B_{\tilde X}((\xi,\zeta),r)$ be a ball belonging to this family that contains infinitely many points of the form $(1-\frac{1}{n},1)$, where $n \in \mathbb N$. Then, $\abs{\xi-1+\frac{1}{n}}<r$ and $\abs{\zeta-1}<r$ for infinitely many $n \in \mathbb N$. Using continuity of the absolute value, we infer that for all but finitely many $m \in \mathbb N$ we have $\abs{\zeta-1+\frac{1}{m}}<r$. This implies that the open ball 
$B_{\tilde X}((\xi,\zeta),r)$ contains (at least) two points $(1-\frac{1}{n_1},1-\frac{1}{m_1})$ and $(1-\frac{1}{n_2},1-\frac{1}{m_2})$ with $n_1<m_1$ and $n_2>m_2$. And then
\[
 \abs[\Bigg]{\frac{f\bigl(1-\frac{1}{n_1}\bigr) - f\bigl(1-\frac{1}{m_1}\bigr)}{\abs[\big]{\frac{1}{m_1}-\frac{1}{n_1}}} - \frac{f\bigl(1-\frac{1}{n_2}\bigr) - f\bigl(1-\frac{1}{m_2}\bigr)}{\abs[\big]{\frac{1}{m_2}-\frac{1}{n_2}}}}=2.
\]
This shows that replacing the open sets $U_1,\ldots,U_N$ with open balls in the condition (LDS) in the statement of Theorem~\ref{thm:compact_in_Lip2} is not possible, even if $E=\mathbb R$ and the set $A$ is a singleton.

A similar example works also in the case of the condition (L) and Theorem~\ref{thm:compact_in_Lip}. It suffices to take the set $A:=\{f,g\}$ with the function $f$ defined above and $g(x)=0$ for $x \in X$ as well as the points $(1-\frac{1}{n},1)$ and $(1-\frac{1}{n}, 1+\frac{1}{n})$. (Note that for a sufficiently large $n \in \mathbb N$ those points must lie in the same open ball belonging to a given family of open balls covering $\tilde X$.) The details are left to the reader.
\end{example}

We end this section with a brief discussion of precompactness criteria in $\BLip_\varphi(X,E)$. As the default norm $\vnorm{\cdot}_\varphi$ in this space is stronger than the norm $\norm{\cdot}_\varphi$ inherited from $\Lip_\varphi(X,E)$, we will need an additional condition.

\begin{theorem}\label{thm:compact_in_BLip}
A non-empty subset $A$ of $\BLip_\varphi(X,E)$ is precompact if and only if
 the set $A-A$ satisfies both the conditions (L) and (B), and the sections $A(x)$ are precompact for every $x \in X$.
\end{theorem}

\begin{proof}[Sketch of the proof]
Consider the linear operator $T\colon \BLip_\varphi(X,E)\to C(X,E)\times C(\tilde{X},E)$ given by $T(f)=(f, \Phi(f))$; here $\Phi$ is the de Leeuw's map. When we endow the target space with the norm $\norm{(g,h)}:=\norm{g}_\infty + \norm{h}_\infty$, then it is clear that $T$ is an isometric embedding of $\Lip_\varphi(X,E)$ into $C(X,E)\times C(\tilde{X},E)$. The rest of the proof is analogous to the proof of Theorem~\ref{thm:compact_in_Lip} with obvious changes.
\end{proof}

In a similar vein to Theorem~\ref{thm:compact_in_Lip2} we get the following compactness criterion for $\BLip_\varphi(X,\mathbb K^n)$. Once again we skip the proof, because at this point it should be evident.

\begin{theorem}\label{thm:compact_in_BLip2}
Let $n \in \mathbb N$ and $\mathbb K \in \{\mathbb R, \mathbb C\}$. A non-empty subset $A$ of $\BLip_\varphi(X,\mathbb K^n)$ is precompact if and only if it is bounded and satisfies the conditions (DS) and (LDS).
\end{theorem}

Finally, we will show that in Theorems~\ref{thm:compact_in_BLip} and~\ref{thm:compact_in_BLip2} the additional conditions on the sets $A-A$ and $A$, respectively, are essential.

\begin{example}
Let $X:=\mathbb R$ and choose $x_\ast=0$ as the base point. For each $n \in \mathbb N$, let $f_n \colon X \to \mathbb R$ be equal to $f_n(x)=1-\abs{\frac{1}{n}x-1}$ on $[0,2n]$, and zero elsewhere. Also, let $A:=\dset{f_n}{n \in \mathbb N}\cup \{f_0\}$, where $f_0$ denotes the zero function. If by $\norm{\cdot}_1$ we denote the norm in $\Lip(X,\mathbb R)$, then $\norm{f_n}_1 \to 0$ as $n \to +\infty$. Thus, $A$ is precompact as a subset of $\Lip(X,\mathbb R)$, that is, with respect to the norm (metric) inherited from $\Lip(X,\mathbb R)$. In particular, the sets $A-A$ and $A$ satisfy the conditions (L) and (LDS), respectively (see Theorems~\ref{thm:compact_in_Lip} and~\ref{thm:compact_in_Lip2}). Furthermore, $A(x)\subseteq [0,1]$ for $x \in X$, which means that all the sections $A(x)$ are precompact.

Now, our aim is to show that the sets $A-A$ and $A$ do not satisfy the conditions (B) and (DS), respectively. Because (DS) is stronger than (B) and is preserved when passing from a set to its algebraic difference, we need to focus on the condition (B) only. Let $U_1,\ldots,U_N$ be an arbitrary finite cover of $X$. Then, there is a set $U_i$ that contains infinitely many positive integers. Among those numbers, we can always find (at least) two of the form $n$ and $n+k$ for some integer $k \geq n$. And so, $\abs[\big]{\abs{(f_n-f_0)(n)}-\abs{(f_n-f_0)(n+k)}}=1$. Thus, the set $A-A$ does not satisfy the condition (B).

Finally, we show that although $A$ is bounded in $\BLip(X,E)$, it is not precompact in the norm of this space $\vnorm{\cdot}_1$; here $\BLip(X,E)$ denotes the subset of $\Lip(X,\mathbb R)$ consisting of bounded functions. Let $(f_{n_k})_{k \in \mathbb N}$ be an arbitrary subsequence of $(f_n)_{n\in \mathbb N}$. For a fixed $N \in \mathbb N$ choose $l \in \mathbb N$ so that $n_l \geq 2n_N$. Then, $\vnorm{f_{n_l}-f_{n_N}}_1 \geq \abs{f_{n_l}(n_l)-f_{n_N}(n_l)}=1$. Thus, $(f_n)_{n\in \mathbb N}$ does not admit a Cauchy subsequence with respect to $\vnorm{\cdot}_1$. Hence, $A$ is not precompact as a subset of $\BLip(X,\mathbb R)$. 
\end{example}

\subsection{Compactness criteria -- compact domains}

This time we assume that the metric space $X$ is compact. In such a case the spaces $\Lip_\varphi(X,E)$ and $\BLip_\varphi(X,E)$ coincide as sets. Moreover, their norms $\norm{\cdot}_\varphi$ and $\vnorm{\cdot}_\varphi$ are equivalent. Hence, we can pass from one to the other without any concern.

Our main goal in this section will be to show that for compact domains all that really matters is the behaviour of the difference quotients appearing in the condition (L) near the diagonal of $X\times X$. First, however, we will prove two technical lemmas. 

\begin{lemma}\label{lem:introGH}
Let $X$ be a compact metric space. Moreover, assume that $A$ is a non-empty and bounded subset of $\Lip_\varphi(X,E)$ such that the sections $A(x)$ are precompact for $x \in X$. For a fixed $\delta>0$ let $\tilde X_\delta:=(X\times X)\setminus \bigcup_{x\in X}(B_X(x,\delta)\times B_X(x,\delta))$. Also, let $\tilde{A}_\delta$ consist of all the mappings $h$ which are the restrictions of $\Phi(f)$ to $\tilde X_\delta$, where $f \in A$ and $\Phi$ is the de Leeuw's map. Then, for any $\varepsilon>0$ there is a finite open cover of $\tilde X_\delta$ such that $\norm{h(x,y)-h(\xi,\eta)}_E\leq \varepsilon$ for any $h\in \tilde{A}_\delta$ and any $(x,y),(\xi,\eta)$ contained in the same element of the cover of $\tilde X_\delta$. 
\end{lemma}

\begin{proof} 
Clearly, $\tilde X_\delta$ is a compact metric space in the maximum metric $d_\infty$ inherited from $X\times X$, and $\tilde A_\delta \subseteq C(\tilde X_\delta,E)$. Furthermore, observe that precompactness of $A(x)$ for $x\in X$ implies the precompactness of $\tilde A_\delta (x,y):=\dset{h(x,y)}{h \in \tilde A_\delta}$ for $(x,y) \in \tilde X_\delta$. Thus, by Corollary~\ref{rem:DS_B_3} we only need to show that the algebraic difference $\tilde A_\delta - \tilde A_\delta$ satisfies the condition (B) with open sets (cf.~Remarks~\ref{rem:open_closed} and~\ref{rem:open_closed2}).

By assumption, the set $A$ is bounded in $\Lip_\varphi(X,E)$. So, there is a positive constant $M$ such that $\vnorm{f}_\varphi=\norm{f}_\infty+\abs{f}_\varphi \leq M$ for $f\in A$. We may also assume that $M\geq \sup\dset{\varphi(d(x,y))}{x,y\in X}$. Because the diagonal $\dset{(x,x)}{x\in X}$ is compact and disjoint from the compact set $\tilde X_\delta$, there is a~positive constant $m$ such that $m\leq \inf\dset{\varphi(d(x,y))}{(x,y)\in \tilde X_\delta}$. 

Now, fix an arbitrary $\varepsilon>0$. Recall that the comparison function is continuous on the non-negative half-axis $[0,+\infty)$ and $\varphi(0)=0$. Hence, there exists $r>0$ such that $\varphi(r) \leq \frac{m^2}{8M^2}\varepsilon$ and $\abs{\varphi(t)-\varphi(s)}\leq \frac{m^2}{8M}\varepsilon$ for any $t,s \in [0,\diam X]$ with $\abs{t-s}\leq 2r$. As the family $U_1,\ldots,U_N$ let us choose any finite covering of $\tilde X_\delta$ with open balls in $\tilde X_\delta$ of radius $\frac{1}{2}r$. Now, take any points $(x,y),(\xi,\eta) \in \tilde X_\delta$ belonging to the same member of the covering and any $h_1,h_2 \in \tilde A_\delta$. Then, $h_1=\Phi(f)|_{\tilde X_\delta}$ and $h_2=\Phi(g)|_{\tilde X_\delta}$ for some $f,g \in A$. Setting for simplicity $d_\varphi:=\varphi \circ d$, we obtain
\begin{align*}
&\abs[\Big]{\norm[\big]{(h_1-h_2)(x,y)}_E - \norm[\big]{(h_1-h_2)(\xi,\eta)}_E}\\
&\ \ = \abs[\Bigg]{\frac{\norm[\big]{(f-g)(x)-(f-g)(y)}_E}{d_\varphi(x,y)}-\frac{\norm[\big]{(f-g)(\xi)-(f-g)(\eta)}_E}{d_\varphi(\xi,\eta)}}\\
&\ \ \leq m^{-2} \abs[\Big]{\, d_\varphi(\xi,\eta)\cdot\norm[\big]{(f-g)(x)-(f-g)(y)}_E - d_\varphi(x,y)\cdot \norm[\big]{(f-g)(\xi)-(f-g)(\eta)}_E\,}\\
&\ \ \leq m^{-2}d_\varphi(\xi,\eta)\cdot \norm[\big]{(f-g)(x)-(f-g)(\xi) + (f-g)(\eta)-(f-g)(y)}_E\\
&\qquad + m^{-2}\abs[\big]{d_\varphi(x,y)-d_\varphi(\xi,\eta)}\cdot \norm[\big]{(f-g)(\xi)-(f-g)(\eta)}_E\\
&\ \ \leq  m^{-2}d_\varphi(\xi,\eta)\cdot \norm[\big]{(f-g)(x)-(f-g)(\xi)}_E + m^{-2}d_\varphi(\xi,\eta)\cdot \norm[\big]{(f-g)(\eta)-(f-g)(y)}_E\\
&\qquad + m^{-2}\abs[\big]{d_\varphi(x,y)-d_\varphi(\xi,\eta)}\cdot \norm[\big]{(f-g)(\xi)}_E + m^{-2}\abs[\big]{d_\varphi(x,y)-d_\varphi(\xi,\eta)}\cdot \norm[\big]{(f-g)(\eta)}_E\\
&\ \ \leq 2m^{-2}M d_\varphi(\xi,\eta)\bigl[d_\varphi(x,\xi) + d_\varphi(\eta,y)\bigr] + 4m^{-2}M\abs[\big]{d_\varphi(x,y)-d_\varphi(\xi,\eta)}\\
&\ \ \leq 2m^{-2}M^2\bigl[\varphi(d(x,\xi)) + \varphi(d(\eta,y))\bigr] + 4m^{-2}M\abs[\big]{\varphi(d(x,y))-\varphi(d(\xi,\eta))}.
\end{align*} 
Recall that $\tilde X_\delta$ is endowed with the maximum metric $d_\infty$. Hence, $\max\{d(x,\xi),d(\eta,y)\}\break=d_\infty((x,y),(\xi,\eta)) \leq r$ and $\abs{d(x,y)-d(\xi,\eta)}\leq d(x,\xi)+d(y,\eta)\leq 2r$. Thus, we obtain
\begin{align*}
 \abs[\Big]{\norm[\big]{(h_1-h_2)(x,y)}_E - \norm[\big]{(h_1-h_2)(\xi,\eta)}_E} \leq \tfrac{1}{2}\varepsilon  + \tfrac{1}{2}\varepsilon =\varepsilon.
\end{align*}
Consequently, the algebraic difference $\tilde A_\delta - \tilde A_\delta$ satisfies the condition (B).
\end{proof}

\begin{lemma}\label{lem:tube}
Let $X$ be a compact metric space. Moreover, assume that for some $n \in \mathbb N$ the diagonal of $X\times X$ is covered by a finite collection of open balls ${B}_X(x_i,\frac{1}{n})\times {B}_X(x_i,\frac{1}{n})$, where $i=1,\ldots, m$. Then, there is a number $\delta>0$ such that $\bigcup_{x\in X}({B}_X(x,\delta)\times {B}_X(x,\delta)) \subseteq \bigcup_{i=1}^m({B}_X(x_i,\frac{1}{n})\times{B}_X(x_i,\frac{1}{n}))$.
\end{lemma}

\begin{proof}
Let $n \in \mathbb N$ such that $\dset{(x,x)}{x \in X} \subseteq \bigcup_{i=1}^m ({B}_X(x_i,\frac{1}{n})\times {B}_X(x_i,\frac{1}{n}))$ be fixed. Suppose that the claim is not true, that is, for any $q\in \mathbb{N}$  there is a point $(\xi^q,\eta^q)\in {B}_X(\zeta^q,\frac{1}{q})\times {B}_X(\zeta^q,\frac{1}{q})$ such that $(\xi^q,\eta^q) \notin \bigcup_{i=1}^m({B}_X(x_i,\frac{1}{n})\times {B}_X(x_i,\frac{1}{n}))$. By compactness of $X$ we may assume that the sequence of centers $(\zeta^q)_{q \in \mathbb N}$ converges to a point $z\in X$. As the set $\bigcup_{i=1}^m (B_X(x_i,\frac{1}{n})\times B_X(x_i,\frac{1}{n}))$ is open in $X\times X$, there is some $\delta>0$ such that $B_X(z,\delta)\times B_X(z,\delta)\subseteq  \bigcup_{i=1}^m({B}_X(x_i,\frac{1}{n})\times {B}_X(x_i,\frac{1}{n}))$. But then, for all but finitely many $q \in \mathbb N$ we have ${B}_X(\zeta^q,\frac{1}{q})\times {B}_X(\zeta^q,\frac{1}{q})\subseteq B_X(z,\delta)\times B_X(z,\delta)$, which is impossible.
\end{proof}

We saw in the proof of Lemma~\ref{lem:introGH} that compactness of $X$ was essential. Without it we would not have been able to apply Corollary~\ref{rem:DS_B_3}, or find the constants $m$ and $M$. On the other hand, in Lemma~\ref{lem:tube} this requirement seems a little bit artificial. Therefore, it is interesting to ask whether Lemma~\ref{lem:tube} would still hold, if the domain $X$ was allowed to be non-compact. Note that this question makes sense only for bounded metric spaces. Because, otherwise, the diagonal of $X\times X$ cannot be covered by a finite family of (open) balls. The following example shows that, if we replace ``compactness'' with ``boundedness'' in the statement of Lemma~\ref{lem:tube}, the result will fail. 

\begin{example}
Set $X:=\mathbb Z \setminus\{0\}$. Now, we define a metric on $X$. To simplify the formulas, we will write $(\pm k, \pm l)$ if both the numbers $k,l$ have the same sign, and $(\pm k, \mp l), (\mp k, \pm l)$ otherwise. Let
\begin{align*}
d(\pm n,\pm n)&:=0 & & \hspace*{-2cm} \text{for $n\in \N$,}\\
d(\pm 1,\mp 1)&:=\tfrac{1}{2}& &\\
d(\pm 1,\pm n)&:=d(\pm n,\pm 1):=\tfrac{1}{2}-\tfrac{1}{2n}&  & \hspace*{-2cm} \text{for $n\geq 2$,}\\
d(\pm 1,\mp n)&:=d(\mp n,\pm 1):=\tfrac{1}{2}&  &\hspace*{-2cm}  \text{for $n\geq 2$,}\\
d(\pm n,\mp m)&:=\tfrac{1}{2n}+\tfrac{1}{2m} & &\hspace*{-2cm}  \text{for $n, m\geq 2$,}\\
d(\pm n,\pm m)&:=\tfrac{1}{2n}+\tfrac{1}{2m} & &\hspace*{-2cm}  \text{for $n,m\geq 2$ with $n\neq m$.} 
\end{align*}
The proof that $d$ is indeed a metric on $X$ is quite straightforward, but tedious. So, we skip it.

Observe now that the diagonal of $X\times X$ is covered by $(B_X(1,\frac{1}{2})\times B_X(1,\frac{1}{2})) \cup (B_X(-1,\frac{1}{2}) \times B_X(-1,\frac{1}{2}))$. Further, given any fixed $\delta>0$ let $m,n \in \mathbb N$ be so that $d(m,-n)<\delta$. Then, clearly $(m,-n) \in B_X(m,\delta)\times B_X(m,\delta)$. But $(m,-n) \notin (B_X(1,\frac{1}{2})\times B_X(1,\frac{1}{2})) \cup (B_X(-1,\frac{1}{2}) \times B_X(-1,\frac{1}{2}))$. This shows that $\bigcup_{k \in X} (B_X(k,\delta)\times B_X(k,\delta)) \not\subseteq (B_X(1,\frac{1}{2})\times B_X(1,\frac{1}{2})) \cup (B_X(-1,\frac{1}{2}) \times B_X(-1,\frac{1}{2}))$ for any $\delta>0$.
\end{example}

Before proving the main result of this section, for a non-empty set $A\subseteq \Lip_\varphi(X,E)$ let us introduce yet another condition:
\begin{enumerate}
 \item[($\Lambda$)] for every $\varepsilon>0$ and $n \in \mathbb N$ there is a radius $\delta>0$ and a finite number of open subsets $U_1,\ldots,U_N$ of $\tilde X$ with $\bigcup_{x\in X} (B_X(x,\delta)\times B_X(x,\delta))\cap\tilde X\subseteq \bigcup_{i=1}^N U_i\subseteq \bigcup_{x\in X} (B_X(x,\frac{1}{n})\times B_X(x,\frac{1}{n}))\cap\tilde X$ such that for any $f \in A$ and $i\in \{1,\ldots,N\}$ we have
\[
\sup_{(x,y)\in U_i} \frac{\norm{f(x)-f(y)}_E}{\varphi(d(x,y))}-\inf_{(x,y)\in U_i}\frac{\norm{f(x)-f(y)}_E}{\varphi(d(x,y))}\leq \varepsilon.
\]
\end{enumerate}
The above condition is a localized versions of (L) in a sense that instead of considering a covering of the whole $\tilde{X}$, we focus our attention on the diagonal of $X\times X$. Note also that here we use a slightly different formulation of the inequality appearing in (L). This change is sanctioned by the fact that for a bounded real function $g$ defined on a non-empty set $Y$ we have $\sup_{x,y \in Y}\abs{g(x)-g(y)}=\sup_{x \in Y} g(x) - \inf_{y \in Y}g(y)$ (see~\cite{Lojasiewicz}*{p.~4}).

And now the main result follows.

\begin{proposition}\label{lem:GH}
Let $X$ be a compact metric space. Moreover, assume that $A$ is a non-empty subset of $\Lip_\varphi(X,E)$ such that the sections $A(x)$ are precompact for $x \in X$. Then, the following conditions are equivalent\textup:
\begin{enumerate}[label=\textup{(\roman*)}]
 \item\label{it:L} $A-A$ satisfies (L) with open sets $U_1,\ldots,U_N$,
 \item\label{it:Lambda} $A-A$ satisfies ($\Lambda$).
\end{enumerate}
\end{proposition}

\begin{proof}
$\ref{it:L}\Rightarrow\ref{it:Lambda}$ Fix any $\varepsilon>0$ and $n \in \mathbb N$. Then, there exists an open cover $U_1,\ldots,U_N$ of $\tilde X$ such that for any $i \in\{1,\ldots,N\}$ and any $f,g \in A$ we have
\[ 
\sup_{(x,y)\in U_i}\frac{\norm{(f-g)(x) - (f-g)(y)}_E}{\varphi(d(x,y))} -\inf_{(x,y) \in U_i}\frac{\norm{(f-g)(x) - (f-g)(y)}_E}{\varphi(d(x,y))}\leq \varepsilon
\]
(see the observation before Proposition~\ref{lem:GH}). As the set $\tilde X$ is open in $X\times X$, so are the members of the collection $U_1,\ldots,U_N$. Since the diagonal $\dset{(x,x)}{x\in X}$ is a compact subset of the metric space $X\times X$, there is a finite number of points $x_1,\ldots,x_k \in X$ such that $\dset{(x,x)}{x\in X}\subseteq \bigcup_{j=1}^k (B_X(x_j,{\frac{1}{n}})\times B_X(x_j,{\frac{1}{n}}))$. For each $i\in \{1,\ldots,N\}$ and $j\in \{1,\ldots,k\}$ set $V_{i,j}:=(B_X(x_j,{\frac{1}{n}})\times B_X(x_j,{\frac{1}{n}}))\cap U_i$. By Lemma~\ref{lem:tube} it is clear that the family consisting of all the non-empty sets $V_{i,j}$ satisfies all the requirements of the condition ($\Lambda$) for $A-A$.

$\ref{it:Lambda}\Rightarrow\ref{it:L}$ We will divide this part into two steps. We begin with showing that, under the assumption of precompactness of the sections $A(x)$, the condition~\ref{it:Lambda} implies that the set $A$ is bounded in $\Lip_\varphi(X,E)$. Let $\delta>0$ and $U_1,\ldots, U_N$ be chosen as in the condition ($\Lambda$) for $A-A$ and $\varepsilon=n=1$.

Suppose that $\sup_{f\in A}\sup_{x\in X}\norm{f(x)}_E=+\infty$. Then, there exist two sequences $(x_m)_{m\in \mathbb{N}}$ in $X$ and $(f_m)_{m\in \mathbb{N}}$ in $A$ for which $\lim_{m\to\infty}\norm{f_m(x_m)}_E=+\infty$. Since the metric space $X$ is compact we may assume that $\lim_{m\to\infty}x_m=y\in X$. If $x_{m_l}=y$ for infinitely many indices $l \in \mathbb N$, we would get $\lim_{l\to\infty}\norm{f_{m_l}(y)}_E=+\infty$. This, in turn, would mean that $A(y)$ is not precompact, which is absurd. Therefore, we may assume that $x_m \neq y$ for all $m \in \mathbb N$. Hence, $(y,x_m) \in (B_X(y,\delta) \times B_X(y,\delta))\cap \tilde X \subseteq \bigcup_{j=1}^N U_j$ for all but finitely many $m \in \mathbb N$. As there are only $N$ sets $U_j$, this implies that there is an index $i \in \{1,\ldots,N\}$ such that $(y,x_m)\in U_i$ for infinitely many $m$'s. Passing to a subsequence yet another time, we may state that $(y,x_m)\in U_i$ for all $m \in \mathbb N$. Now, pick a function $f \in A$ and a point $(\xi,\eta) \in U_i$. Note that, because of the compactness of the metric space $X$ and the continuity of the functions $\varphi$ and $f$, there is a constant $R>0$ with $\varphi(d(x_m,y))\leq R$ and $\norm{f(x_m)}_E\leq R$ for $m \in \mathbb N$. Observe also that, since the sections $A(x)$ are precompact, (adjusting the constant $R$ if necessary) we have $\norm{(f_m-f)(z)}_E\leq R$ for $m \in \mathbb N$, where $z \in \{\xi,\eta,y\}$. Thus, for all $m \in \mathbb N$ we get
\begin{align*}
&\norm{f_m(x_m)}_E\\
 &\ \ \leq \norm{f_m(x_m)-f(x_m)}_E+\norm{f(x_m)}_E\\
 &\ \ \leq \norm{(f_m-f)(x_m)-(f_m-f)(y)}_E+\norm{(f_m-f)(y)}_E+\norm{f(x_m)}_E\\
 &\ \ = \varphi(d(x_m,y))\cdot \frac{\norm{(f_m-f)(x_m)-(f_m-f)(y)}_E}{\varphi(d(x_m,y))}+\norm{(f_m-f)(y)}_E+\norm{f(x_m)}_E\\
 &\ \ \leq  \varphi(d(x_m,y)) \cdot \sup_{(p,q)\in U_i}\frac{\norm{(f_m-f)(p)-(f_m-f)(q)}_E}{\varphi(d(p,q))}+\norm{(f_m-f)(y)}_E+\norm{f(x_m)}_E\\
 &\ \ \leq \varphi(d(x_m,y))\biggl(1+\frac{\norm{(f_m-f)(\xi)-(f_m-f)(\eta)}_E}{\varphi(d(\xi,\eta))}\biggr)+\norm{(f_m-f)(y)}_E+\norm{f(x_m)}_E\\
&\ \ \leq 3R + 2R^2/\varphi(d(\xi,\eta))<+\infty.
\end{align*}
This leads to a contradiction. Hence, $M:=\sup_{f\in A}\norm{f}_\infty<+\infty$.

Using a similar reasoning to the above one, we will show that the set $A$ is also bounded in the Lipschitz semi-norm $\abs{\cdot}_\varphi$. (Note that we will use the same letters as in the former part, but with possibly different meanings.) Suppose that
\[
 \sup_{f\in A}\sup_{(x,y)\in \tilde X} \frac{\norm{f(x)-f(y)}_E}{\varphi(d(x,y))}=+\infty.
\]
Hence, there are sequences $(x_m, y_m)_{m\in \mathbb{N}}$ in $\tilde X$ and $(f_m)_{m\in \mathbb{N}}$ in $A$ for which $\lim_{m\to\infty} \norm{f_m(x_m)-f_m(y_m)}_E/\varphi(d(x_m,y_m))=+\infty$. Because the metric space $X \times X$ is compact, we may assume that $\lim_{m \to \infty}(x_m,y_m)=(a,b) \in X\times X$. If $a\neq b$, then for every $m \in \mathbb N$ we would have
\[
 \frac{\norm{f_m(x_m)-f_m(y_m)}_E}{\varphi(d(x_m,y_m))} \leq \frac{2M}{\varphi(d(x_m,y_m))}.
\]
Consequently,
\[
 \lim_{m\to\infty} \frac{\norm{f_m(x_m)-f_m(y_m)}_E}{\varphi(d(x_m,y_m))}\leq \frac{2M}{\varphi(d(a,b))}<+\infty,
\]
which is absurd. Thus, $a=b$. This means that $(x_m,y_m) \in (B_X(a,\delta)\times B_X(a,\delta))\cap \tilde{X} \subseteq \bigcup_{j=1}^N U_j$ for all but finitely many $m \in \mathbb N$. Therefore, there is an index $i \in \{1,\ldots,N\}$ such that $(x_m,y_m) \in U_i$ for infinitely many $m$'s. Passing to a subsequence if necessary, we may actually claim that the above relation holds for all $m \in \mathbb N$. If we fix $f \in A$ and $(\xi,\eta)\in U_i$, then in view of the condition ($\Lambda$) for all $m \in \mathbb N$ we have
\begin{align*}
 \frac{\norm{f_m(x_m)-f_m(y_m)}_E}{\varphi(d(x_m,y_m))}& \leq \frac{\norm{(f_m-f)(x_m)-(f_m-f)(y_m)}_E}{\varphi(d(x_m,y_m))}+\frac{\norm{f(x_m)-f(y_m)}_E}{\varphi(d(x_m,y_m))}\\
 & \leq \sup_{(p,q)\in U_i}\frac{\norm{(f_m-f)(p)-(f_m-f)(q)}_E}{\varphi(d(p,q))}+\frac{\norm{f(x_m)-f(y_m)}_E}{\varphi(d(x_m,y_m))}\\
 & \leq 1+\frac{\norm{(f_m-f)(\xi)-(f_m-f)(\eta)}_E}{\varphi(d(\xi,\eta))}+\frac{\norm{f(x_m)-f(y_m)}_E}{\varphi(d(x_m,y_m))}\\
 & \leq 1 + 4M/\varphi(d(\xi,\eta)) + \abs{f}_{\varphi}<+\infty.
\end{align*}
This is impossible.  Therefore, $\sup_{f\in A} \abs{f}_{\varphi}<+\infty$. Consequently, the set $A$ is bounded in $\Lip_\varphi(X,E)$. 

Now, we focus on the second -- main -- step of this part of the proof. (Once again we ``reset'' our notation.) Fix $\varepsilon>0$ and $n \in \mathbb N$. Then, by the condition~\ref{it:Lambda}, there is a radius $\delta>0$ and an open cover $U_1,\ldots,U_k$ of $\tilde X$ with $\bigcup_{x\in X} (B_X(x,\delta)\times B_X(x,\delta))\cap\tilde X\subseteq \bigcup_{i=1}^k U_i\subseteq \bigcup_{x\in X} (B_X(x,\frac{1}{n})\times B_X(x,\frac{1}{n}))\cap\tilde X$ such that for any $f,g\in A$ and $i\in \{1,\ldots,k\}$ we have
\[
\sup_{(x,y)\in U_i} \frac{\norm{(f-g)(x)-(f-g)(y)}_E}{\varphi(d(x,y))}-\inf_{(x,y)\in U_i}\frac{\norm{(f-g)(x)-(f-g)(y)}_E}{\varphi(d(x,y))}\leq \varepsilon.
\]
From the previous step we know that the set $A$ is bounded in $\Lip_{\varphi}(X,E)$. Hence, we can apply Lemma~\ref{lem:introGH} with $\frac{1}{2}\varepsilon$ and $\frac{1}{4}\delta$ to find a finite open cover $W_{k+1},\ldots, W_N$ of $\tilde X_{\frac{1}{4}\delta}$ such that 
\[
 \sup_{f \in A} \norm[\Bigg]{\frac{f(x)-f(y)}{\varphi(d(x,y))} - \frac{f(\xi)-f(\eta)}{\varphi(d(\xi,\eta))}}_E \leq \frac{1}{2}\varepsilon
\]
for any $(x,y),(\xi,\eta)$ belonging to the same member of the collection $W_{k+1},\ldots, W_N$. For every $i=k+1,\ldots,N$ let $U_i:=W_i\setminus \dset{(x,y) \in X\times X}{\inf_{z \in X}\max\{d(x,z),d(y,z)\}\leq \frac{1}{2}\delta}$. We may clearly assume that all the sets $U_i$ are non-empty; otherwise we just remove the empty ones from our family.

We claim that the sets $U_i$ for $i \in \{k+1,\ldots,N\}$ are open in $X \times X$, and hence in $\tilde X$. So, let us fix a set $U_j$ with $j \in \{k+1,\ldots,N\}$  and take any point $(a,b) \in U_j$. Since $(a,b) \in W_j$, and the set $W_j$ is open in $\tilde X_{\frac{1}{4}\delta}$, there is a radius $r>0$ such that $(B_X(a,r) \times B_X(b,r)) \cap \tilde X_{\frac{1}{4}\delta}\subseteq W_j$. Set $\rho:=\inf_{z \in X}\max\{d(a,z),d(b,z)\}$ and $R:=\min\{r,\frac{1}{4}\delta,\rho-\frac{1}{2}\delta\}$. Note that by definition we have $\rho>\frac{1}{2}\delta$. So, the radius $R$ is well-defined. We will show now that the open ball $B_{X}(a,R)\times B_X(b,R)$ in $X\times X$ is included in $U_j$. Clearly, it is included in $B_X(a,r) \times B_X(b,r)$. Furthermore, $B_{X}(a,R)\times B_X(b,R) \subseteq \tilde X_{\frac{1}{4}\delta}$. Otherwise, there would exist points $u,p,q \in X$ such that $d(a,p)<R$, $d(b,q)<R$, $d(p,u)<\frac{1}{4}\delta$ and $d(q,u)<\frac{1}{4}\delta$. This, in turn, would imply that $d(u,a)\leq d(a,p)+d(p,u)<\frac{1}{2}\delta$. And, similarly, $d(u,b)<\frac{1}{2}\delta$. But this is impossible, because $\rho= \inf_{z \in X}\max\{d(a,z),d(b,z)\}>\frac{1}{2}\delta$. Therefore, $B_{X}(a,R)\times B_X(b,R) \subseteq (B_X(a,r) \times B_X(b,r))\cap \tilde X_{\frac{1}{4}\delta} \subseteq W_j$. 

It remains to show that the sets $\dset{(x,y) \in X\times X}{\inf_{z \in X}\max\{d(x,z),d(y,z)\}\leq \frac{1}{2}\delta}$ and $B_{X}(a,R)\times B_X(b,R)$ are disjoint. Suppose, on the contrary, that there is a point $(\alpha,\beta) \in  B_{X}(a,R)\times B_X(b,R)$ such that $\inf_{z \in X}\max\{d(\alpha,z),d(\beta,z)\}\leq \frac{1}{2}\delta$. As the function $z \mapsto \max\{d(\alpha,z),d(\beta,z)\}$ is continuous and the metric space $X$ is compact, there is $\zeta\in X$ such that $\inf_{z \in X}\max\{d(\alpha,z),d(\beta,z)\}=\max\{d(\alpha,\zeta),d(\beta,\zeta)\}$. Now, we have
\begin{align*}
\rho&=\inf_{z \in X}\max\{d(a,z),d(b,z)\}\\
&\leq \max\{d(a,\zeta),d(b,\zeta)\}\\
&\leq \max\{d(a,\alpha),d(b,\beta)\} + \max\{d(\alpha,\zeta),d(\beta,\zeta)\}\\
&< \rho-\tfrac{1}{2}\delta + \tfrac{1}{2}\delta = \rho.
\end{align*}
This is absurd. Thus, the sets $\dset{(x,y) \in X\times X}{\inf_{z \in X}\max\{d(x,z),d(y,z)\}\leq \frac{1}{2}\delta}$ and $B_{X}(a,R)\times B_X(b,R)$ are disjoint. Consequently, $B_{X}(a,R)\times B_X(b,R) \subseteq U_j$. Hence, the sets $U_{k+1},\ldots,U_N$ are open in $\tilde{X}$.

The rest of the proof is straightforward. It is not difficult to see that the family of sets $U_1,\ldots, U_k, U_{k+1},\ldots, U_N$ satisfies all the requirements of the condition~\ref{it:L}.
\end{proof}

From Remark~\ref{rem:open_closed_in_L_LDS}, Theorem~\ref{thm:compact_in_Lip} and Proposition~\ref{lem:GH} we immediately get the following second compactness criterion in the space of Lipschitz continuous functions $\Lip_\varphi(X,E)$.

\begin{theorem}\label{thm:compact_in_Lip'}
Let $X$ be a compact metric space. A non-empty subset $A$ of $\Lip_\varphi(X,E)$ is precompact if and only if the set $A-A$ satisfies the condition ($\Lambda$) and the sections $A(x)$ for $x \in X$ are precompact.
\end{theorem} 

\begin{remark}
Theorem~\ref{thm:compact_in_Lip'} provides also a compactness criterion in the space $\Lip_0(X,E)$ with $X$ being a compact metric space (cf.~Remark~\ref{rem:Lip_0}).
\end{remark}

In Corollary~\ref{rem:DS_B_3} we saw that in the case of $C(X,E)$, when the domain $X$ is compact, we can replace the assumption ``$A-A$ satisfies the condition (B)'' with ``$A$ satisfies the condition (DS)''. The conditions (L) and (LDS), we consider in this section, can be viewed as versions of (B) and (DS) for $\Phi(A)$, respectively. Therefore, a natural question is whether in Theorem~\ref{thm:compact_in_Lip'} instead of the conditions ($\Lambda$) or (L) we can use (LDS)? The answer is negative, as shown by the following simple example.

\begin{example}
Consider the Banach space $l^\infty$ of all bounded real sequences, endowed with the supremum norm $\norm{\cdot}_\infty$. For each $n \in \mathbb N$ let $x_n:=(0,\ldots,0,\frac{1}{n},0,\ldots)$, where the only non-zero element is at the $n$-th place. Moreover, let $x_0$ be the zero sequence. Set $X:=\dset{x_n}{n \in \mathbb N \cup\{0\}}$. Because $\norm{x_n}_\infty \to 0$ when $n \to +\infty$, the metric space $X$ is compact (in the metric inherited from $l^\infty$). Now, for $a \in [0,1]$ define $f_a \colon X \to l^\infty$ by $f_a(x)=ax$. And, consider $A:=\dset{f_a}{a\in [0,1]}$. Clearly, $A\subseteq \Lip(X,l^\infty)$.

Observe that $A(x_0)$ is a singleton, and $A(x_n)$ for every $n \in \mathbb N$ is homeomorphic with the interval $[0,1]$. This means that each section $A(x)$ for $x \in X$ is (pre)compact. Further, note that for any point $(x,y) \in \tilde{X}$ and any $a,b \in [0,1]$ we have
\[
 \frac{\norm{(f_a-f_b)(x)-(f_a-f_b)(y)}_{\infty}}{\norm{x-y}_{\infty}}=\abs{a-b}.
\]
Therefore, the set $A-A$ satisfies the condition (L) with $U_1:=\tilde{X}$. By Proposition~\ref{lem:GH}, it also satisfies the condition ($\Lambda$).

Finally, we claim that the set $A$ does not satisfy the condition (LDS). Let $W_1,\ldots,W_N$ be any cover of $\tilde X$. As there are infinitely many points $(x_n,x_0)$, there is an element $W_i$ of the cover that contains (at least) two of them, say $(x_n,x_0)$ and $(x_m,x_0)$. Then,
\[
 \norm[\Bigg]{\frac{f_1(x_n)-f_1(x_0)}{\norm{x_n-x_0}_\infty} - \frac{f_1(x_m)-f_1(x_0)}{\norm{x_m-x_0}_\infty}}_{\infty} =  \norm[\Bigg]{\frac{x_n}{\norm{x_n}_\infty} - \frac{x_m}{\norm{x_m}_\infty}}_{\infty}=1.
\]
This proves our claim.

As a closing remark, note that the above example works only because the target space $l^\infty$ is infinite-dimensional. If its dimension were finite, it might not be possible to find a set $W_i$ containing two different points $(x_n,x_0)$ and $(x_m,x_0)$.
\end{example}

\subsection{Compactness criterion in $\lip_\varphi(X,E)$}
\label{sec:55}

In the theory of Lipschitz spaces special attention is paid to the subclass $\lip_\varphi(X,E)$ of $\Lip_\varphi(X,E)$ called the \emph{little Lipschitz space}. It, along with its various modifications and versions, appears naturally, when the problem of the second predual of $\Lip_\varphi(X,E)$ is investigated. Although it is possible to define little Lipschitz space with $X$ being non-compact (Weaver does it in Section~4.2 of his book~\cite{weaver}), it requires much more technicalities. Since $\lip_\varphi(X,E)$ is not the main object of our study, and the precompactness criterion for this space heavily relies on the results from the previous section, in the sequel we will assume that the metric space $X$ is compact.    

Let us now introduce the little Lipschitz space properly. Fix a comparison function $\varphi$. Also, let $(X, d)$ be a compact metric space and
let $(E,\norm{\cdot}_E)$ be a normed space. The class $\lip_\varphi(X,E)$ consists of all functions from $\Lip_\varphi(X,E)$ that satisfy the following condition
\begin{equation*}\label{eq:little_lipschitz_def}
 \lim_{\delta\to 0^+} \sup\dset[\Bigg]{\frac{\norm{f(x)-f(y)}_E}{\varphi(d(x,y))}}{\text{$x,y \in X$ with $0<d(x,y) \leq \delta$}} = 0.
\end{equation*}
Equivalently, a map $f \in \Lip_{\varphi}(X,E)$ belongs to $\lip_\varphi(X,E)$, if for every $\varepsilon>0$ there exists $\delta>0$ such that $\norm{f(x)-f(y)}_E\leq \varepsilon \varphi(d(x,y))$ for all $x,y \in X$ with $d(x,y) \leq \delta$. Mappings satisfying the above equivalent conditions are called \emph{locally flat}. 

The space $\lip_\varphi(X,E)$ is endowed with one of the equivalent norms $\norm{\cdot}_\varphi$ or $\vnorm{\cdot}_\varphi$ inherited from $\Lip_\varphi(X,E)$. Further, it is a Banach space, when $E$ is complete. It is easy to check that $\lip_\varphi(X,E)$ is a closed subspace of $\Lip_\varphi(X,E)$. However, in general, $\lip_\varphi(X,E)\neq \Lip_\varphi(X,E)$; to see this it suffices to set $\varphi(t)=t$, $X:=[0,1]$, $E:=\mathbb R$ and take the function $f(x)=x$. Actually, it turns out that in this setting the little Lipschitz space consists only of the constant functions (see~\cite{weaver}*{Example~4.8}). What is surprising is that this phenomenon holds more generally. If $X$ is a compact connected Riemannian manifold and $\varphi(t)=t$, then $\lip_\varphi(X,\mathbb R)$ is one-dimensional and consists of constant functions only (see~\cite{weaver}*{Example~4.9}).

On the other hand, for any comparison function $\varphi$ such that $\lim_{t \to 0^+} \varphi(t)/t=+\infty$, the class $\Lip(X,\mathbb R)$ is dense in $\lip_\varphi(X, \mathbb R)$ (see~\cite{hanin}*{Proposition~4} and cf.~\cite{weaver2}*{Corollary~1.5}).

Now, we move to characterizing precompact subsets of the little Lipschitz space $\lip_\varphi(X,E)$. This time there is no need to distinguish between $E$ being finite- or infinite-dimensional. There is a very simple explanation of this phenomenon. Namely, if $f \in \lip_{\varphi}(X,E)$, then the mapping $\Phi(f)$, where $\Phi$ is the de Leeuw's map, can be extended continuously from $\tilde X$ over the whole (compact) product $X \times X$ simply by putting $\Phi(f)(x,x)=0$ for $x \in X$. 

\begin{theorem}\label{th:little:lip}
Let $X$ be a compact metric space. A non-empty subset $A$ of $\lip_\varphi(X,E)$ is precompact if and only if
\begin{enumerate}[label=\textup{(\roman*)}]
 \item the sections $A(x)$ are precompact for every $x \in X$, and
 \item the set $A$ is \emph{uniformly locally flat}, that is, for every $\varepsilon>0$ there exists $\delta>0$ such that for any $x,y \in X$ with $d(x,y)\leq \delta$ we have $\norm{f(x)-f(y)}_E \leq \varepsilon \varphi(d(x,y))$ for every $f\in A$.
\end{enumerate}
\end{theorem}

\begin{proof}
Fix $\varepsilon>0$. Since $A$ is a uniformly locally flat set of functions,  there exists $\eta>0$ such that $\norm{f(x)-f(y)}_E\leq \frac{1}{2}\varepsilon \varphi(d(x,y))$ for all $f \in A$ and all $x,y \in X$ with $d(x,y)\leq \eta$. (To make it easier to refer to the previous results, we have changed the notation a little bit here.) Furthermore, let us fix $n \in \mathbb N$, and choose $m \in \mathbb N$ so that $\frac{1}{m}\leq \min\{\frac{1}{2}\eta,\frac{1}{n}\}$. Now, pick a finite family of points $x_1,\ldots,x_N \in X$ such that the diagonal of $X\times X$ is covered by the union of the balls  $B_X(x_i,\frac{1}{m})\times B_X(x_i,\frac{1}{m})$, where $i\in \{1,\ldots,N\}$.
 Define $U_i:=(B_X(x_i,\frac{1}{m})\times B_X(x_i,\frac{1}{m}))\cap \tilde X$ for $i\in \{1,\ldots,N\}$. Then, clearly $\bigcup_{i=1}^N U_i \subseteq \bigcup_{x\in X}(B_X(x,\frac{1}{n})\times B_X(x,\frac{1}{n}))\cap \tilde X$. Moreover, by Lemma~\ref{lem:tube} (applied with $m$) there is a radius $\delta>0$ such that $\bigcup_{x\in X}(B_X(x,\delta)\times B_X(x,\delta))\cap \tilde X \subseteq \bigcup_{i=1}^N U_i$. If we fix $f,g \in A$ and $i\in \{1,\ldots,N\}$, then by the uniform local flatness of $A$ we obtain
\begin{align*}
& \sup_{(x,y)\in U_i} \frac{\norm{(f-g)(x)-(f-g)(y)}_E}{\varphi(d(x,y))}-\inf_{(x,y)\in U_i}\frac{\norm{(f-g)(x)-(f-g)(y)}_E}{\varphi(d(x,y))}\\
&\qquad \leq \sup_{(x,y)\in U_i} \frac{\norm{(f-g)(x)-(f-g)(y)}_E}{\varphi(d(x,y))} \leq \varepsilon.
\end{align*}
Thus, by Theorem~\ref{thm:compact_in_Lip'} we can conclude that $A$ is precompact in $\Lip_\varphi(X,E)$, and consequently in $\lip_\varphi(X,E)$.

To prove the opposite implication, let us assume that $A$ is a precompact subset of $\lip_\varphi(X,E)$ and $f_1,\ldots,f_m\in A$ is its finite $\frac{1}{2}\varepsilon$-net. Choose $\delta>0$ so that $\max_{1\leq i\leq m}\norm{f_i(x)-f_i(y)}_E\leq \frac{1}{2}\varepsilon\varphi(d(x,y))$ for $x,y \in X$ with $d(x,y)\leq \delta$. For any $f\in A$ there is some $j\in \{1,\ldots,m\}$ for which we have $\norm{f-f_j}_\varphi \leq \frac{1}{2}\varepsilon$. Hence,
\begin{align*}
\norm{f(x)-f(y)}_E &\leq \norm{(f-f_j)(x)-(f-f_j)(y)}_E+\norm{f_j(x)-f_j(y)}_E\\
& \leq \tfrac{1}{2}\varepsilon\varphi(d(x,y))+ \tfrac{1}{2}\varepsilon\varphi(d(x,y))=\varepsilon\varphi(d(x,y)),
\end{align*}
whenever $x,y \in X$ are such that $d(x,y)\leq\delta$. This shows that the family $A$ is uniformly locally flat. To end the proof it suffices to observe that the Lipschitz norm $\norm{\cdot}_\varphi$ is stronger than the supremum norm $\norm{\cdot}_\infty$, and hence precompactness of $A$ implies precompactness of each section $A(x)$, where $x \in X$.
\end{proof}

\begin{remark}
Theorem~\ref{th:little:lip} can be also applied to the space $\lip_0(X,E)$, which consists of those functions from $\Lip_0(X,E)$ that are locally flat.
\end{remark}

\begin{remark}
Theorem~\ref{th:little:lip} in the case when $E:=\mathbb R$ and $\varphi(t)=t$ was proven by Johnson using a different approach than ours (see~\cite{johnson}*{Theorem~3.2}). A similar result was also obtained by Garc\'{\i}a-Lirola \emph{et al.}, who, motivated by some considerations of duality, studied compactness in a certain little Lipschitz space (see~\cite{GPR}*{Lemma~2.7}). To be more specific, they worked in a subspace of $\lip_0(X,\mathbb R)$ consisting of those functions that are also continuous with respect to a topology $\tau$ on $X$. Furthermore, they assumed that $X$ is compact with respect to $\tau$ and that the metric $d$ of $X$ is $\tau$-lower semi-continuous.
\end{remark}

\subsection*{Acknowledgement}
We thank Luis Garc\'{\i}a-Lirola, Colin Petitjean and Abraham Rueda Zoca for bringing their paper~\cite{GPR} to our attention. We also thank the anonymous referee for his/her valuable comments. Firstly, they inspired us to provide more
in-depth discussion of the compactness criteria for spaces of Lipschitz mappings presented in Section~5. And secondly, they helped us to improve the exposition and readability of the paper.


\begin{bibdiv}
\begin{biblist}

\bib{AKORR}{book}{
  author={Agarwal, R. P.},
  author={Karap\i nar, E.},
  author={O'Regan, D.},
  author={Rold\'{a}n-L\'{o}pez-de-Hierro, A. F.},
  title={Fixed point theory in metric type spaces},
  publisher={Springer, Cham},
  date={2015},
}

\bib{alexiewicz1969analiza}{book}{
  title={Analiza funkcjonalna \textup(Functional analysis\textup)},
  author={Alexiewicz, A.},
  number={vol. 49},
  series={Monografie Matematyczne},
  year={1969},
  publisher={Pa{\'n}stwowe Wydawn. Naukowe, Warszawa},
  language={in Polish}
}

  \bib{Ambrosetti}{article}{
     author={Ambrosetti, A.},
     title={Un teorema di esistenza per le equazioni differenziali negli spazi di Banach},
     year={1967},
     journal={Rend. Sem. Mat. Univ. Padova},
     volume={39},
     pages={349--360},
   }

\bib{banas_nalepa}{article}{
author = {Bana\'s, J.},
author = {Nalepa, R.},
title = {On the space of functions with growths tempered by a modulus of continuity and its applications},
volume = {2013},
journal = {Journal of Function Spaces},
year = {2013},
number = {01}
}

\bib{banas_nalepa2}{article}{
author = {Bana\'s, J.},
author = {Nalepa, R.},
title = {A measure of noncompactness in the space of functions with tempered increments on the half-axis and its applications},
journal = {Journal of Mathematical Analysis and Applications},
volume = {474},
number = {2},
pages = {1551--1575},
year = {2019},
}



\bib{BW}{article}{
   author={Berninger, H.},
	 author={Werner, D.},
   title={Lipschitz spaces and M-ideals},
	 journal={Extracta Math.},
	 volume={18},
	 pages={33--56},
	 date={2003},
}

\bib{BBK}{book}{
  author={Borkowski, M.},
  author={Bugajewska, D.},
  author={Kasprzak, P.},
  title={Selected problems in nonlinear analysis},
	publisher={The Nicolaus Copernicus University Press},
	address={Toru\'n},
	date={2021},
}

\bib{cobzas2001}{article}{
author = {Cobza{\c{s}}, {\c{S}}.},
year = {2001},
number = {01},
pages = {9--14},
title = {Compactness in spaces of Lipschitz functions},
volume = {30},
journal = {Revue d’Analyse Numérique et de Théorie de l’Approximation}
}

\bib{cobzas_book}{book}{
  title={Lipschitz functions},
  author={Cobza{\c{s}}, {\c{S}}.},
  author={Miculescu, R.},
  author={Nicolae, A.},
  series={Lecture Notes in Mathematics},
  year={2019},
  publisher={Springer International Publishing}
}

\bib{Conway}{book}{
   author={Conway, J. B.},
   title={A course in functional analysis},
   series={Graduate Texts in Mathematics},
   volume={96},
   edition={2},
   publisher={Springer-Verlag, New York},
   date={1990},
}

\bib{DS}{book}
{
title={Linear operators: general theory},
author={Dunford, N.},
author={Schwartz, J. T.},
  series={Pure and Applied Mathematics},
  year={1958},
  publisher={Interscience Publishers},
}

\bib{evans}{book}{
  title={Partial differential equations},
  author={Evans, L.C.},
  series={Graduate studies in mathematics},
  year={2010},
  publisher={American Mathematical Society}
}

\bib{eveson}{article}{
   author={Eveson, S. P.},
   title={Compactness criteria and measures of noncompactness in function
   spaces},
   journal={J. London Math. Soc. (2)},
   volume={62},
   date={2000},
   number={1},
   pages={263--277},
}

\bib{GPR}{article}{
   author={Garc\'{\i}a-Lirola, L.},
   author={Petitjean, C.},
   author={Rueda Zoca, A.},
   title={On the structure of spaces of vector-valued Lipschitz functions},
   journal={Studia Math.},
   volume={239},
   date={2017},
   number={3},
   pages={249--271},
}


\bib{gilbarg_trudinger}{book}{
author ={Gilbarg, D.},
author={Trudinger, N. S.},
title={Elliptic partial differential equations of second order}, 
publisher={Springer-Verlag, New York/Berlin}, 
year={1983},
}

\bib{GV}{book}{
   author={Gorenflo, R.},
   author={Vessella, S.},
   title={Abel integral equations. Analysis and applications},
   series={Lecture Notes in Mathematics},
   volume={1461},
   publisher={Springer-Verlag, Berlin},
   date={1991},
}

\bib{GKM}{article}{
   author={Gulgowski, J.},
	 author={Kasprzak, P.},
	 author={Ma\'ckowiak, P.},
   title={Compactness in normed spaces: a unified approach through semi-norms},
   journal={Topol. Methods Nonlinear Anal.},
	 note={(in press)},
}

\bib{hanin}{article}{
   author={Hanin, L. G.},
   title={Kantorovich-Rubinstein norm and its application in the theory of
   Lipschitz spaces},
   journal={Proc. Amer. Math. Soc.},
   volume={115},
   date={1992},
   number={2},
   pages={345--352},
}
	
	
\bib{johnson}{article}{
   author={Johnson, J. A.},
   title={Banach spaces of Lipschitz functions and vector-valued Lipschitz
   functions},
   journal={Trans. Amer. Math. Soc.},
   volume={148},
   date={1970},
   pages={147--169},
}	


\bib{kreyszig}{book}{
   author={Kreyszig, E.},
   title={Introductory functional analysis with applications},
   publisher={John Wiley \& Sons, New York-London-Sydney},
   date={1978},
}

\bib{deLeeuw}{article}{
	 author={de Leeuw, K.},
	 title={Banach spaces of Lipschitz functions},
	 journal={Studia Math.},
	 volume={21}, 
	 pages={55--66},
	 date={1961/1962},
 }

\bib{Lojasiewicz}{book}{
   author={\L ojasiewicz, S.},
   title={An introduction to the theory of real functions},
   series={A Wiley-Interscience Publication},
   edition={3},
   note={With contributions by M. Kosiek, W. Mlak and Z. Opial;
   translated from the Polish by G. H. Lawden;
   translation edited by A. V. Ferreira},
   publisher={John Wiley \& Sons, Ltd., Chichester},
   date={1988},
}


\bib{MV}{book}{
  author={Meise, R.},
  author={Vogt, D.},
  title={Introduction to functional analysis},
  series={Oxford Graduate Texts in Mathematics},
  volume={2},
  publisher={The Clarendon Press, Oxford University Press, New York},
  date={1997},
}




\bib{phillips}{article}{
   author={Phillips, R. S.},
   title={On linear transformations},
   journal={Trans. Amer. Math. Soc.},
   volume={48},
   date={1940},
   pages={516--541},
}
		

\bib{Saiedinezhad}{article}{
author={Saiedinezhad, S.},
title = {On a measure of noncompactness in the Holder space $C^{k,\gamma}(\Omega)$ and its application},
journal = {Journal of Computational and Applied Mathematics},
volume = {346},
pages = {566--571},
year = {2019},
}


\bib{samko}{book}{
author = {Samko, S. G.},
author ={Kilbas, A. A.},
author={Marichev, O. I.}, 
title={Fractional Integrals and Derivatives: Theory and Applications}, 
publisher = {Gordon and Breach, New York}, 
year = {1993},
}

\bib{weaver2}{article}{
   author={Weaver, N.},
   title={Subalgebras of little Lipschitz algebras},
   journal={Pacific J. Math.},
   volume={173},
   date={1996},
   number={1},
   pages={283--293},
}

\bib{weaver}{book}{
author = {Weaver, N.},
title = {Lipschitz Algebras},
publisher = {World Scientific Publishing Co. Pte. Ltd},
address={Singapore},
year = {2018},
}



\end{biblist}
\end{bibdiv}
\end{document}